\documentclass[a4paper,11pt]{article}
\usepackage[pdftex]{graphicx,xcolor}
\usepackage{float}
\usepackage{amsmath,amssymb,bm,amsthm}
\usepackage{ascmac}
\usepackage{array}
\usepackage{hyperref}
\usepackage[hang,small,bf]{caption}
\usepackage[subrefformat=parens]{subcaption}
\usepackage{natbib}
\usepackage{enumitem}

\bibpunct[:]{(}{)}{,}{a}{}{,}
\hypersetup{
  colorlinks=true,
	citecolor=blue,
	linkcolor=red,
	urlcolor=orange,
}
\title{Pathwise Representation of the Smoothing Distribution in Continuous-Time Linear Gaussian Models}
\author{Masahiro Kurisaki\thanks{Graduate School of Mathematical Sciences, the University of Tokyo\protect\\ email: \texttt{makurisaki@g.ecc.u-tokyo.ac.jp}}~\thanks{Japan Science and Technology Agency CREST \protect\\ The author was supported by Japan Science and Technology Agency CREST JPMJCR2115, and JSPS KAKENHI Grant Number JP24KJ0667.}}

\numberwithin{equation}{section}
\allowdisplaybreaks[3]
\newtheorem{theorem}{Theorem}[section]
\newtheorem{proposition}[theorem]{Proposition}
\newtheorem{lemma}[theorem]{Lemma}

\theoremstyle{definition}

\theoremstyle{remark}
\newtheorem{remark}{Remark}

\begin{document}
\maketitle
\begin{abstract}
  We study the filtering and smoothing problem for continuous-time linear Gaussian systems. While classical approaches such as the Kalman–Bucy filter and the Rauch–Tung–Striebel (RTS) smoother provide recursive formulas for the conditional mean and covariance, we present a pathwise perspective that characterizes the smoothing error dynamics as an Ornstein–Uhlenbeck process. As an application, we show that standard filtering and smoothing equations can be uniformly derived as corollaries of our main theorem. In particular, we provide the first mathematically rigorous derivation of the Bryson–Frazier smoother in the continuous-time setting. Beyond offering a more transparent understanding of the smoothing distribution, our formulation enables pathwise sampling from it, which facilitates Monte Carlo methods for evaluating nonlinear functionals.
\end{abstract}

\section{Introduction}

Estimating the hidden states of a stochastic dynamical system from noisy observations is a fundamental problem in various fields, including signal processing, control theory, and mathematical finance. In this paper, we consider the following continuous-time linear Gaussian model:
\begin{align}
\label{eq1-1} dX_t &= a(t)X_tdt + b(t)dV_t,\\
\label{eq1-2} dY_t &= c(t)X_tdt + \sigma(t)dW_t,
\end{align}
where $V$ and $W$ are independent Brownian motions, and $a$, $b$, $c$, and $\sigma$ are known deterministic functions. The process $X$ is unobserved, and the goal is to infer its behavior based on the observed process $Y$.

The problem of estimating the conditional expectation $E[X_t \mid \mathcal{Y}_t]$, where $\mathcal{Y}_t = \sigma(Y_s : 0 \leq s \leq t)$ is the $\sigma$-field generated by the observations up to time $t$, is known as the \emph{filtering problem}. In the linear Gaussian case, it is solved by the classical Kalman–Bucy filter, whose derivations are well documented for both discrete-time models \citep{kalman1960,Sarkka_2013} and continuous-time settings \citep{kalman1961,Kailath1,bain2009fundamentals,Liptser2001}.

A related problem is the estimation of $E[X_s \mid \mathcal{Y}_t]$ for $0 \leq s \leq t$, known as the \emph{smoothing problem}. Its solution is given by the Rauch–Tung–Striebel (RTS) smoother, with derivations available in both discrete \citep{Rauch1965,Sarkka_2013} and continuous time \citep{Liptser2001-2,doi:10.1137/S0040585X97T99037X,RUTKOWSKI1993377}. The conventional approach to smoothing focuses on deriving recursive formulas for the conditional mean and covariance.

In this paper, we take a different perspective on the smoothing problem by providing a pathwise characterization of the conditional distribution of the hidden process $\{X_s\}_{0 \leq s \leq t}$ given $\mathcal{Y}_t$. Rather than focusing only on the conditional mean and covariance, we explicitly describe the dynamics of the smoothing error process, which turns out to satisfy a linear stochastic differential equation. In particular, we show that the smoothing error evolves as an Ornstein–Uhlenbeck-type process driven by a Brownian motion that is independent under the conditional measure.

This viewpoint—interpreting the smoothing error as a stochastic process—has previously appeared in the engineering literature, most notably in \citet{1057207}. However, that work lacks mathematical rigor, relying on formal reasoning without a solid probabilistic foundation, and requires a non-singularity assumption for the conditional covariance. By contrast, our derivation is fully rigorous, based on stochastic analysis which are familiar to probabilists, such as Girsanov's theorem and Itô's formula, and does not rely on any additional regularity assumptions.

Our formulation has several key advantages. First, it provides a more transparent and unified understanding of the smoothing distribution by identifying the underlying stochastic mechanism. As a result, classical results such as the Kalman–Bucy filter and RTS smoother can be derived uniformly as corollaries of our main theorem.

In particular, our approach provides the first mathematically rigorous derivation of the Bryson–Frazier (BF) smoother in the continuous-time setting. The BF smoother, originally proposed in the discrete-time context \citep{bryson1963smoothing,bierman1973fixed}, offers an alternative representation of the RTS smoother. A key advantage of the BF formulation is that it avoids the computation of the inverse of the conditional covariance matrix, which is required in the RTS approach. While a continuous-time version of the BF smoother has been formally derived in the engineering literature \citep{1102735}, a rigorous probabilistic foundation has been lacking until now.

  It is worth emphasizing that although the BF smoother is mathematically equivalent to the RTS smoother, attempting to derive it from the Kalman–Bucy filter or RTS smoother inevitably leads to expressions involving the inverse of the conditional covariance (see Remark \ref{remark-inverse}). As a result, additional regularity assumptions must be imposed to ensure well-definedness. Our approach circumvents this difficulty by constructing the BF smoother directly and independently of those existing formulations.

A more practical advantage of our formulation is that it enables pathwise sampling from the smoothing distribution via an explicit stochastic representation. This facilitates Monte Carlo approximation of expectations of nonlinear functionals, such as \( E[\sup_{0 \leq s \leq t} X_s \mid \mathcal{Y}_t] \), and the \( Q \)-function in EM algorithms for nonlinear models, when combined with linearization techniques such as the extended Kalman filter.

The rest of the paper is organized as follows. Section~\ref{section-heuristic-argument} illustrates the main idea using a simple one-dimensional example. Section~\ref{section-linear-filtering} establishes the general formula for the smoothing dynamics in the linear Gaussian model. In Section~\ref{section-linear-equations}, we derive classical filtering and smoothing equations as corollaries of our result and discuss their implications. Finally, Section~\ref{section-application} outlines further applications of the theorem.

\section{One-dimensional Case}\label{section-heuristic-argument}
Before addressing the general problem, we illustrate the main idea using a one-dimensional model, given by
\begin{align*}
  Y = cX + \sigma Z,
\end{align*}
where \( X \sim N(\mu_x, \sigma_x^2) \) and \( Z \sim N(0, 1) \) are independent random variables, with \( \sigma_x, \sigma > 0 \). Furthermore, we assume $c\neq 0$. In this model, the conditional expectation of \( X \) given \( Y \) can be expressed as
\begin{align}
  E[X | Y = y] &= \frac{\displaystyle \int_{\mathbb{R}} x \exp\left( -\frac{1}{2\sigma^2}(cx - y)^2 \right) p(x) \, dx}{\displaystyle \int_{\mathbb{R}} \exp\left( -\frac{1}{2\sigma^2}(cx - y)^2 \right) p(x) \, dx}\nonumber\\
  &\label{eq3-14} =\frac{\displaystyle E\left[ X\exp\left( -\frac{c^2X^2}{2\sigma^2}+\frac{cXy}{\sigma^2} \right) \right]}{\displaystyle E\left[ \exp\left( -\frac{c^2X^2}{2\sigma^2}+\frac{cXy}{\sigma^2} \right) \right]}
\end{align}
where \( p(x) \) is the probability density function of \( X \). This formula is the analogue of (\ref{Kallianpur-Striebel-linear}).

This final expression can be further rewritten as $\frac{\sigma^2}{c}\frac{d}{dy} \log E\left[ \exp\left( -\frac{c^2X^2}{2\sigma^2}+\frac{cXy}{\sigma^2} \right) \right]$. Due to the normality of \( X \), it can be shown that the log part is a quadratic function of \( y \). Therefore, we can write
\begin{align*}
  E[X | Y = y]=\frac{\displaystyle E\left[ X\exp\left( -\frac{c^2X^2}{2\sigma^2}+\frac{cXy}{\sigma^2} \right) \right]}{\displaystyle E\left[ \exp\left( -\frac{c^2X^2}{2\sigma^2}+\frac{cXy}{\sigma^2} \right) \right]}=Ay+B.
\end{align*}
By taking expectation of both sides (with respect to $y$), we obtain $B=E[X]-AcE[X]$ and hence
\begin{align}
  \label{eq3-25}E[X | Y = y]=\frac{\displaystyle E\left[ X\exp\left( -\frac{c^2X^2}{2\sigma^2}+\frac{cXy}{\sigma^2} \right) \right]}{\displaystyle E\left[ \exp\left( -\frac{c^2X^2}{2\sigma^2}+\frac{cXy}{\sigma^2} \right) \right]}=E[X]+A(y-cE[X]).
\end{align}
Furthermore, by differentiating (\ref{eq3-25}), $A$ can be determined as
\begin{align}
  A=&\frac{d}{dy}\frac{\displaystyle E\left[ X\exp\left( -\frac{c^2X^2}{2\sigma^2}+\frac{cXy}{\sigma^2} \right) \right]}{\displaystyle E\left[ \exp\left( -\frac{c^2X^2}{2\sigma^2}+\frac{cXy}{\sigma^2} \right) \right]}\nonumber\\
  \label{eq3-26}\begin{split}
    =&\frac{c}{\sigma^2}\frac{\displaystyle  E\left[ X^2\exp\left( -\frac{c^2X^2}{2\sigma^2}+\frac{cXy}{\sigma^2} \right) \right]}{\displaystyle E\left[ \exp\left( -\frac{c^2X^2}{2\sigma^2}+\frac{cXy}{\sigma^2} \right) \right]} - \frac{c}{\sigma^2} \left( \frac{\displaystyle E\left[ X\exp\left( -\frac{c^2X^2}{2\sigma^2}+\frac{cXy}{\sigma^2} \right) \right]}{\displaystyle E\left[ \exp\left( -\frac{c^2X^2}{2\sigma^2}+\frac{cXy}{\sigma^2} \right) \right]} \right)^2
  \end{split}  
\end{align}
Since $A$ does not depend on $y$, we can set $y=0$ in the final expression, which yields
\begin{align}
  \label{eq3-27}A=\frac{c}{\sigma^2} \left\{ \frac{\displaystyle E\left[ X^2 \exp\left( -\frac{c^2X^2}{2\sigma^2} \right) \right]}{\displaystyle E\left[ \exp\left( -\frac{c^2X^2}{2\sigma^2} \right) \right]} - \left( \frac{\displaystyle E\left[ X \exp\left( -\frac{c^2X^2}{2\sigma^2} \right) \right]}{\displaystyle E\left[ \exp\left( -\frac{c^2X^2}{2\sigma^2} \right) \right]} \right)^2 \right\}.
\end{align}
Therefore, (\ref{eq3-25}) ends up in the conditional mean formula
\begin{align}
  \label{eq3-28}\begin{split}
    &E[X | Y = y] = E[X] \\
    &+ \frac{c}{\sigma^2} \left\{ \frac{\displaystyle E\left[ X^2 \exp\left( -\frac{c^2X^2}{2\sigma^2} \right) \right]}{\displaystyle E\left[ \exp\left( -\frac{c^2X^2}{2\sigma^2} \right) \right]} - \left( \frac{\displaystyle E\left[ X \exp\left( -\frac{c^2X^2}{2\sigma^2} \right) \right]}{\displaystyle E\left[ \exp\left( -\frac{c^2X^2}{2\sigma^2} \right) \right]} \right)^2 \right\} (y - cE[X]).
  \end{split}  
\end{align}

Moreover, recalling the Bayesian formula (\ref{eq3-14}), we notice that (\ref{eq3-26}) is equivalent to
\begin{align*}
  A=\frac{c}{\sigma^2}(E[X^2|Y=y]-E[X|Y=y]^2)=\frac{c}{\sigma^2}E[V|Y=y].
\end{align*}
Since $A$ is given by (\ref{eq3-27}), we obtain the conditional variance formula
\begin{align}
  \label{eq3-29}V[X|Y=y]=\frac{\displaystyle E\left[ X^2 \exp\left( -\frac{c^2X^2}{2\sigma^2} \right) \right]}{\displaystyle E\left[ \exp\left( -\frac{c^2X^2}{2\sigma^2} \right) \right]} - \left( \frac{\displaystyle E\left[ X \exp\left( -\frac{c^2X^2}{2\sigma^2} \right) \right]}{\displaystyle E\left[ \exp\left( -\frac{c^2X^2}{2\sigma^2} \right) \right]} \right)^2.
\end{align}
Further differentiating (\ref{eq3-26}) yields the third and higher cumulants of $X$ given $Y$, but they are zero since $A$ is a constant. In this way, we can conclude that the conditional distribution of $X$ given $Y=y$ is normal with the mean and variance given by (\ref{eq3-28}) and (\ref{eq3-29}), respectively.

The key contribution here is that we proceed by differentiating the conditional expectation, evaluating it at zero, and deriving the non-trivial expressions in (\ref{eq3-28}) and (\ref{eq3-29}), which represent the conditional distribution as a quotient of ordinary expectations. In the infinite-dimensional setting, this formulation enables us to analyse the conditional distribution using classical tools from stochastic calculus, such as It\^o's formula and the Girsanov theorem.

\section{Dynamics of the conditional process}\label{section-linear-filtering}
Now, let us go on to the original continuous-time case. Let $(\Omega, \mathcal{F}, \{\mathcal{F}_t\}_{t \geq 0}, P)$ be a filtered probability space, and assume that the processes \( \{X_t\} \) and \( \{Y_t\} \) are solutions of the following linear equations:
\begin{align}
  \label{eq-linear-system-X} &dX_t = a(t) X_t \, dt +  b(t) \, dV_t, \\
  \label{eq-linear-system-Y} &dY_t = c(t) X_t \, dt + \sigma(t) \, dW_t,~~Y_0=0
\end{align}
where \( a \), \( b \), \( c \), and \( \sigma \) are measurable functions taking values in \( M_{d_1}(\mathbb{R}) \), \( M_{d_1,m_1}(\mathbb{R}) \), \( M_{d_2,d_1}(\mathbb{R}) \), and \( M_{d_2,m_2}(\mathbb{R}) \), respectively. Here, \( \{V_t\} \) and \( \{W_t\} \) are independent \( m_1 \)- and \( m_2 \)-dimensional \( \{\mathcal{F}_t\} \)-Brownian motions. We also assume that \( X_0 \) is normally distributed and independent of \( \{V_t\} \) and \( \{W_t\} \), so that \( \{(X_t, Y_t)\} \) forms a \( d_1 + d_2 \)-dimensional Gaussian process. Additionally, we assume that for every \( t \geq 0 \),
\begin{align}
   &\int_0^t |a(s)| \, ds < \infty,\quad 
   \int_0^t |b(s)|^2 \, ds < \infty, \quad\label{eq-assumption-b}\\
  &\int_0^t |c(s)|^2 \, ds < \infty, \quad
   \int_0^t |\sigma(s)|^2 \, ds < \infty,\nonumber
\end{align}
and that there exists a constant \( C > 0 \) such that for every \( t \geq 0 \),
\begin{align}
  \label{eq-assumption-sigma-2} \lambda_{\min}(\sigma(t) \sigma(t)^\top) > C,
\end{align}
where \( \lambda_{\min}(\sigma(t) \sigma(t)^\top) \) denotes the smallest eigenvalue of \( \sigma(t) \sigma(t)^\top \).
Throughout this paper, $|A|$ denotes the Frobenius norm for a matrix $A$. 

We begin by presenting the continuous-time analogue of the Bayesian formula (\ref{eq3-14}), known in the context of filtering as the Kallianpur–Striebel formula.
\begin{proposition}\label{prop-Kallianpur-Striebel-linear}
  Let $T>0$ and define a probability measure $Q$ by $Q(A)=E\left[ 1_AZ_T^{-1} \right]$ for $A \in \mathcal{F}$, where
  \begin{align}
    \label{eq-def-Z}\begin{split}
      Z_t=&\exp\left( \int_0^tX_s^\top c(s)^\top(\sigma(s)\sigma(s)^\top)^{-1}\sigma(s) dW_s\right.\\
      &\left.+\frac{1}{2}\int_0^tX_s^\top c(s)^\top(\sigma(s)\sigma(s)^\top)^{-1} c(s)X_s ds\right)\\
      =&\exp\left( \int_0^tX_s^\top c(s)^\top(\sigma(s)\sigma(s)^\top)^{-1} dY_s\right.\\
      &\left.-\frac{1}{2}\int_0^tX_s^\top c(s)^\top(\sigma(s)\sigma(s)^\top)^{-1} c(s)X_s ds\right).
    \end{split}    
  \end{align}
  Then, under the new measure $Q$, the process 
  \begin{align*}
    \overline{Y}_t=\int_0^t (\sigma(t)\sigma(t)^\top)^{\frac{1}{2}} dY_s~~(0\leq t\leq T)
  \end{align*}
  is a Brownian motion independent of $\{X_t\}_{0\leq t\leq T}$, and the law of $\{X_t\}_{0\leq t\leq T}$ is the same as its law under $P$. Furthermore, for $0\leq t \leq T$ and a random variable $U$ satisfying $E[|U|]<\infty$, it holds that $E[Z_t|\mathcal{Y}_t]>0~\mathrm{a.s.}$ and 
  \begin{align}
    \label{Kallianpur-Striebel-linear}&E[U|\mathcal{Y}_t]=\frac{\displaystyle E_Q\left[ UZ_t\middle|\mathcal{Y}_t \right]}{\displaystyle E_Q\left[ Z_t\middle|\mathcal{Y}_t \right]}.
  \end{align}  
\end{proposition}
\begin{proof}
  See Exercises 3.11, 3.14 and Proposition 3.16 in \citet{bain2009fundamentals}.
\end{proof}

In the following, we aim to fix $Y$ in the expression for $Z_t$. To this end, let us rewrite $Z_t$ as
\begin{align}
  Z_t=&\exp\left( \int_0^tX_s^\top c(s)^\top(\sigma(s)\sigma(s)^\top)^{-1} dY_s\nonumber\right.\\
  &\left.-\frac{1}{2}\int_0^tX_s^\top c(s)^\top(\sigma(s)\sigma(s)^\top)^{-1} c(s)X_s ds\right)\nonumber\\
  \label{eq-Z-Ito}\begin{split}
    =&\exp\left( X_t^\top \hat{Y}_t- \int_0^t\hat{Y}_s^\top dX_s-\frac{1}{2}\int_0^tX_s^\top c(s)^\top(\sigma(s)\sigma(s)^\top)^{-1} c(s)X_s ds\right),
  \end{split}
\end{align}
where $\displaystyle \hat{Y}_s=\int_0^s c(u)^\top(\sigma(u)\sigma(u)^\top)^{-1}dY_u$, and write
\begin{align}
  \label{eq-def-zt-y}\begin{split}
    z_t(y)=&\exp\left( X_t^\top y(t)- \int_0^ty(s)^\top dX_s\right.\\
    &\left.-\frac{1}{2}\int_0^tX_s^\top c(s)^\top(\sigma(s)\sigma(s)^\top)^{-1} c(s)X_s ds\right)
  \end{split}  
\end{align}
for a bounded measurable function $y:[0,t]\to \mathbb{R}^{d_1}$.

Then, since $X$ and $Y$ are independent and $X$ has the same distribution under both $P$ and $Q$ (Proposition~\ref{prop-Kallianpur-Striebel-linear}), it is natural to expect that
\begin{align}
  E_Q[Z_t \mid \mathcal{Y}_t] = E[z_t(y)]\big|_{y = \hat{Y}}. \label{eq-conditional-of-Z}
\end{align}
This identity does hold, although the argument requires care because the stochastic integral in $Z_t$ is not defined pathwise. A rigorous justification is given through Lemma~\ref{lemma-stochastic-integral-measurable-function}.

Next, we show that (\ref{eq-conditional-of-Z}) is a "quadratic form" of $Y$.
\begin{lemma}\label{lemma-quadratic-form}
  For $t \in [0,T]$ and a bounded measurable function $y:[0,t]\to \mathbb{R}^{d_1}$, we can write
  \begin{align*}
    \log E_Q[z_t(y)]=&h_0(t)+h_1(t)y(t)+y(t)^\top h_2(t)y(t)+\int_0^t y(s)^\top k_1(s;t)ds\\
    &+\int_0^t y(s)^\top k_2(s;t) y(s)ds+\int_0^t y(s)^\top k_3(s;t) y(t)ds
  \end{align*}
  where $\int_0^t |k_i(s;t)|ds<\infty~~(i=1,2,3).$
\end{lemma}
The proof of this lemma is a direct computation of the expectation. For details, see Section~\ref{section-expectation-calculation}. We now establish a formula that reduces the conditional expectation to an ordinary expectation, which corresponds to equations~(\ref{eq3-28}) and~(\ref{eq3-29}) in the one-dimensional case. 
\begin{theorem}\label{theorem-normal-expectation}
  Let $t\geq0$, $s,u \in [0,t]$ and write
  \begin{align}
    \label{def-gamma-s-u-t}\begin{split}
      &\gamma(s,u;t)\\
    &=\frac{\displaystyle E\left[X_s\otimes X_u\exp\left( -\frac{1}{2}\int_0^tX_s^\top c(s)^\top(\sigma(s)\sigma(s)^\top)^{-1} c(s)X_s ds\right)\right]}{\displaystyle E\left[\exp\left( -\frac{1}{2}\int_0^tX_s^\top c(s)^\top(\sigma(s)\sigma(s)^\top)^{-1} c(s)X_s ds\right)\right]}\\
    &-\frac{\displaystyle E\left[X_s\exp\left( -\frac{1}{2}\int_0^tX_s^\top c(s)^\top(\sigma(s)\sigma(s)^\top)^{-1} c(s)X_s ds\right)\right]}{\displaystyle E\left[\exp\left( -\frac{1}{2}\int_0^tX_s^\top c(s)^\top(\sigma(s)\sigma(s)^\top)^{-1} c(s)X_s ds\right)\right]}\\
    &\otimes \frac{\displaystyle E\left[X_u\exp\left( -\frac{1}{2}\int_0^tX_s^\top c(s)^\top(\sigma(s)\sigma(s)^\top)^{-1} c(s)X_s ds\right)\right]}{\displaystyle E\left[\exp\left( -\frac{1}{2}\int_0^tX_s^\top c(s)^\top(\sigma(s)\sigma(s)^\top)^{-1} c(s)X_s ds\right)\right]}.
    \end{split}    
  \end{align}
  Then, the conditional distribution of $\{X_s\}_{0\leq s\leq t}$ given $\mathcal{Y}_t$ is a Gaussian process with mean
  \begin{align}
    \label{eq-mean}E[X_s|\mathcal{Y}_t]= E[X_s] + \int_0^t \gamma(s,u;t) c(u)^\top (\sigma(u) \sigma(u)^\top)^{-1} (dY_u - c(u) E[X_u] \, du),
  \end{align}
  and covariance
  \begin{align}
    \label{eq-cov}\mathrm{Cov}(X_s,X_u|\mathcal{Y}_t)=\gamma(s,u;t).
  \end{align}
\end{theorem}

We observe that equations~(\ref{eq-mean}) and~(\ref{eq-cov}) resemble their one-dimensional counterparts, equations~(\ref{eq3-28}) and~(\ref{eq3-29}). This also clarifies the connection between the Kalman-Bucy filter and the Cameron-Martin formula discussed in Section~7 of \citet{Liptser2001}, both of which involve Riccati equations.

These formulas can intuitively be derived by differentiating the quotient in~(\ref{Kallianpur-Striebel-linear}), with $U = X_s$, with respect to "$dY_u$" (or via the Malliavin derivative), and then evaluating the result at $Y = 0$. However, it is not straightforward to justify these operations rigorously in their current form.

To overcome this mathematical difficulty, we introduce discretized versions of $Z_t$ and the path $\{X_s\}_{0 \leq s \leq t}$ by
\begin{align}
  \label{eq-def-hat-Z}\begin{split}    
  Z_t^{(n)}=&\exp\left( X_t^\top \hat{Y}_{t}- \sum_{i=1}^n \hat{Y}_{s_{i-1}}^\top (X_{s_i}-X_{s_{i-1}}) \right.\\
  &\left.-\frac{1}{2}\int_0^tX_s^\top c(s)^\top(\sigma(s)\sigma(s)^\top)^{-1} c(s)X_s ds\right),
  \end{split}
\end{align}
and
\begin{align}
  \label{eq-def-X-n}X_s^{(n)}=\sum_{i=1}^n X_{s_{i-1}}1_{[s_{i-1},s_{i})}(s)+X_{t}1_{\{t\}}(s)=X_{\frac{t}{n}\left[\frac{s}{t}n\right]}
\end{align}
where $\displaystyle s_i=s_i^n=(i/n)t$. Then the following convergence holds.
\begin{lemma}\label{lemma-3-2}
  For any fixed $u_1,\cdots,u_k \in [0,t]$, it holds that
  \begin{align}
    \label{eq-convergence-frac-Z}\frac{E_Q[X_{u_1}^{(n)}\otimes \cdots \otimes X_{u_{k}}^{(n)}Z_t^{(n)}|\mathcal{Y}_t]}{E_Q[Z_t^{(n)}|\mathcal{Y}_t]}  \xrightarrow{L^2} \frac{E_Q[X_{u_1}\otimes \cdots \otimes X_{u_{k}}Z_t|\mathcal{Y}_t]}{E_Q[Z_t|\mathcal{Y}_t]}~~(n\to \infty).
  \end{align}
\end{lemma}
\begin{proof}
  If we define $\hat{Y}^{(n)}(s) = \sum_{i=1}^n \hat{Y}_{s_{i-1}} 1_{[s_{i-1}, s_i)}(s) + \hat{Y}_t 1_{\{t\}}(s)$, then we have \( Z_t^{(n)} = z_t(\hat{Y}^{(n)}) \), and it follows that
\begin{align*}
  E_Q[Z_t^{(n)} | \mathcal{Y}_t] = E_Q[z_t(y)] \big|_{y = \hat{Y}^{(n)}}
\end{align*}
and
\begin{align*}
  E_Q[X_{u_1}^{(n)} \otimes \cdots \otimes X_{u_k}^{(n)} Z_t^{(n)} | \mathcal{Y}_t] = E_Q[X_{u_1}^{(n)} \otimes \cdots \otimes X_{u_k}^{(n)} z_t(y)] \big|_{y = \hat{Y}^{(n)}}.
\end{align*}
Therefore, the desired result follows from the convergence $\sup_{0 \leq s \leq t} E_Q\left[\left| \hat{Y}_s - \hat{Y}_s^{(n)} \right|^p \right] \to 0$ for \( p \geq 1 \), and the fact that we can write 
\begin{align}
  \label{eq-tensor-expectation}\begin{split}
    &\frac{E_Q[ X_{u_1}\otimes \cdots \otimes X_{u_{n}}z_t(y)]}{E_Q[ z_t(y)]}=\bigotimes_{j=1}^n \left\{ p_{j}(u_j;t)+q_j(u_j;t) y(t)+ \int_0^tr_j(u_j,s;t) y(s)ds \right\}
  \end{split}    
  \end{align}
  where $\int_0^t |r_j(u_j,s;t)|<\infty$ (see Section \ref{section-expectation-calculation}).

\end{proof}

\begin{proof}[Proof of Theorem \ref{theorem-normal-expectation}]
  First, let us write
  \begin{align*}
    &z_t^{(n)}(y_0,\cdots,y_n)
    =\exp\left( X_t^\top y_{n}- \sum_{i=1}^n y_{i-1}^\top (X_{s_i}-X_{s_{i-1}})\right.\\
    &\left. -\frac{1}{2}\int_0^tX_s^\top c(s)^\top(\sigma(s)\sigma(s)^\top)^{-1} c(s)X_s ds\right),
  \end{align*}
  and consider the function $F(y_0,\cdots,y_{n})=E_Q\left[z_t^{(n)}(y_0,\cdots,y_n)\right]$ for $y_0,\cdots,y_{n} \in \mathbb{R}^{d_1}$. Then it is easy to verify that
  \begin{align*}
    &\partial_{y_{i}}F(y_0,\cdots,y_{n})
    =-E_Q\left[(X_{s_{i+1}}-X_{s_{i}})z_t^{(n)}(y_0,\cdots,y_n)\right]~~(i=0,1,\cdots,n),
  \end{align*}
  where $X_{s_{n+1}}=0$ for convenience, and higher derivatives can be calculated in the same manner. 
  Therefore, we have 
  \begin{align}
    &\sum_{i,j=0}^{n}\partial_{y_i}\otimes \partial_{y_j}\log F(y_0,\cdots,y_n)1_{[0,s_{i+1})}(s)1_{[0,s_{j+1})}(u)\nonumber\\
    =&\sum_{i,j=0}^{n}\frac{E_Q[(X_{s_{i+1}}-X_{s_i})\otimes (X_{s_{j+1}}-X_{s_j})z_t^{(n)}(y_0,\cdots,y_n)]}{E_Q[z_t^{(n)}(y_0,\cdots,y_n)]}1_{[0,s_{i+1})}(s)1_{[0,s_{j+1})}(u)\nonumber\\
    &-\sum_{i,j=0}^{n}\frac{E_Q[(X_{s_{i+1}}-X_{s_i})z_t^{(n)}(y_0,\cdots,y_n)]}{E_Q[z_t^{(n)}(y_0,\cdots,y_n)]}1_{[0,s_{i+1})}(s)\nonumber\\
    &\otimes \frac{ E_Q[(X_{s_{j+1}}-X_{s_j})z_t^{(n)}(y_0,\cdots,y_n)]}{E_Q[z_t^{(n)}(y_0,\cdots,y_n)]}1_{[0,s_{j+1})}(u)\nonumber \\
    \label{eq3-11}\begin{split}
      =&\frac{E_Q[X_s^{(n)}\otimes X_u^{(n)}z_t^{(n)}(y_0,\cdots,y_n)]}{E_Q[z_t^{(n)}(y_0,\cdots,y_n)]}\\
      &-\frac{E_Q[X_s^{(n)}z_t^{(n)}(y_0,\cdots,y_n)]\otimes E_Q[X_u^{(n)}z_t^{(n)}(y_0,\cdots,y_n)]}{E_Q[z_t^{(n)}(y_0,\cdots,y_n)]^2},
    \end{split}    
  \end{align}
  where $X^{(n)}$ is defined in (\ref{eq-def-X-n}).

  On the other hand, it follows from Lemma \ref{lemma-quadratic-form} that $\log F(y_1,\cdots,y_{n})$ is a quadratic form of $y_0,\cdots,y_{n}$, and thus $\partial_{y_i}\otimes \partial_{y_j}\log F(y_0,\cdots,y_n)$ does not depend on $y_0,\cdots,y_n$. Thus (\ref{eq3-11}) does not depend on $y_0,\cdots,y_n$, and we have
  \begin{align*}
    &\frac{E_Q[X_s^{(n)}\otimes X_u^{(n)}z_t^{(n)}(y_0,\cdots,y_n)]}{E_Q[z_t^{(n)}(y_0,\cdots,y_n)]}\\
    &-\frac{E_Q[X_s^{(n)}z_t^{(n)}(y_0,\cdots,y_n)]\otimes E_Q[X_u^{(n)}z_t^{(n)}(y_0,\cdots,y_n)]}{E_Q[z_t^{(n)}(y_0,\cdots,y_n)]^2}\\
    =&\frac{E_Q[X_s^{(n)}\otimes X_u^{(n)}z_t^{(n)}(0,\cdots,0)]}{E_Q[z_t^{(n)}(0,\cdots,0)]}\\
    &-\frac{E_Q[X_s^{(n)}z_t^{(n)}(0,\cdots,0)]\otimes E_Q[X_u^{(n)}z_t^{(n)}(0,\cdots,0)]}{E_Q[z_t^{(n)}(0,\cdots,0)]^2}.
  \end{align*}
  Substituting $(\hat{Y}_{s_0},\cdots,\hat{Y}_{s_{n}})$ into $(y_0,\cdots,y_n)$, we obtain
  \begin{align}
      &\frac{E_Q[X_s^{(n)}\otimes X_u^{(n)}Z_t^{(n)}|\mathcal{Y}_t]}{E_Q[Z_t^{(n)}|\mathcal{Y}_t]}
      -\frac{E_Q[X_s^{(n)}Z_t^{(n)}|\mathcal{Y}_t]\otimes E_Q[X_u^{(n)}Z_t^{(n)}|\mathcal{Y}_t]}{E_Q[Z_t^{(n)}|\mathcal{Y}_t]^2}\nonumber\\
    =&\frac{\displaystyle E\left[X_s^{(n)}\otimes X_u^{(n)}\exp\left( -\frac{1}{2}\int_0^tX_s^\top c(s)^\top(\sigma(s)\sigma(s)^\top)^{-1} c(s)X_s ds\right)\right]}{\displaystyle E\left[\exp\left( -\frac{1}{2}\int_0^tX_s^\top c(s)^\top(\sigma(s)\sigma(s)^\top)^{-1} c(s)X_s ds\right)\right]}\nonumber\\
    &-\frac{\displaystyle E\left[X_s^{(n)}\exp\left( -\frac{1}{2}\int_0^tX_s^\top c(s)^\top(\sigma(s)\sigma(s)^\top)^{-1} c(s)X_s ds\right)\right]}{\displaystyle E\left[\exp\left( -\frac{1}{2}\int_0^tX_s^\top c(s)^\top(\sigma(s)\sigma(s)^\top)^{-1} c(s)X_s ds\right)\right]}\nonumber\\
    \label{eq3-24}&\otimes \frac{\displaystyle E\left[X_u^{(n)}\exp\left( -\frac{1}{2}\int_0^tX_s^\top c(s)^\top(\sigma(s)\sigma(s)^\top)^{-1} c(s)X_s ds\right)\right]}{\displaystyle E\left[\exp\left( -\frac{1}{2}\int_0^tX_s^\top c(s)^\top(\sigma(s)\sigma(s)^\top)^{-1} c(s)X_s ds\right)\right]}
  \end{align}
  since $z_t^{(n)}(\hat{Y}_{s_0},\cdots,\hat{Y}_{s_{n}})=Z_t^{(n)}$, as introduced in (\ref{eq-def-hat-Z}). Here, we used the distribution of $\{X_t\}_{0\leq t\leq T}$ is the same under $E$ and $E_Q$. By Lemma \ref{lemma-3-2}, and Proposition \ref{prop-Kallianpur-Striebel-linear}, the left-hand side of this converges in $L^2$ to 
  \begin{align*}
    &\frac{E_Q[X_s\otimes X_uZ_t|\mathcal{Y}_t]}{E_Q[Z_t|\mathcal{Y}_t]}-\frac{E_Q[X_sZ_t|\mathcal{Y}_t]\otimes E_Q[X_uZ_t|\mathcal{Y}_t]}{E_Q[Z_t|\mathcal{Y}_t]^2}
    =\mathrm{Cov}(X_s,X_u|\mathcal{Y}_t)\nonumber
  \end{align*}
  as $n \to \infty$. On the other hand, the $L^2$-convergence of the right-hand side directly follows from the $L^2$-convergence of $X_s^{(n)}$. Therefore, we obtain (\ref{eq-cov}).

  In the same way, since the third and higher derivatives of $\log F(y_0,\cdots,y_n)$ are zero, we have
  \begin{align*}
    0=&\sum_{i_1,\cdots,i_k=0}^{n}\partial_{y_{i_1}}\otimes \cdots\otimes \partial_{y_{i_k}}\log F(\hat{Y}_{s_0},\cdots,\hat{Y}_{s_{n}})1_{[0,s_{i_1}]}(u_1)\cdots 1_{[0,s_{i_k}]}(u_k)
  \end{align*}
  for $k\geq 3$. The right-hand side converges in probability to the $k$-th cumulant of $X_{u_1},\cdots,X_{u_k}$, and thus it must be almost surely zero, which means that $\{X_s\}_{0\leq s\leq t}$ is conditionally Gaussian.

  To prove (\ref{eq-mean}), let us consider
  \begin{align}
     &\sum_{i=0}^{n}\partial_{y_i}\log F(\hat{Y}_{s_0},\cdots,\hat{Y}_{s_{n}})1_{[0,s_{i+1})}(s)=\frac{E_Q[X_s^{(n)}Z_t^{(n)}|\mathcal{Y}_t]}{E_Q[Z_t^{(n)}|\mathcal{Y}_t]}, \label{eq3-10}
  \end{align}
  and let $D$ be the Malliavin derivative with respect to the Brownian motion $\overline{Y}_s=\int_0^s (\sigma(s)\sigma(s)^\top)^{-\frac{1}{2}}dY_s$ on $Q$. Since $\log F(y_0,\cdots,y_n)$ is a quadratic form, $\partial_{y_i}\log F(y_0,\cdots,y_n)$ is linear. Hence, by the definition of the Malliavin derivative of smooth random variables \citep{nualart2006malliavin} and $\hat{Y}_t=\int_0^t c(s)^\top (\sigma(s)\sigma(s)^\top)^{-\frac{1}{2}}d\overline{Y}_s,$ the Malliavin derivative of (\ref{eq3-10}) is
  \begin{align*}
    &D_u\frac{E_Q[X_s^{(n)}Z_t^{(n)}|\mathcal{Y}_t]}{E_Q[Z_t^{(n)}|\mathcal{Y}_t]}\\
    =&\sum_{i=0}^{n}(\sigma(u)\sigma(u)^\top)^{-\frac{1}{2}}c(u)\partial_{y_j}\partial_{y_i}\log F(\hat{Y}_{s_0},\cdots,\hat{Y}_{s_{n}})1_{[0,s_i]}(s)1_{[0,s_j]}(u)\\
    =&(\sigma(u)\sigma(u)^\top)^{-\frac{1}{2}}c(u)\left\{ \frac{E[X_u^{(n)}\otimes X_s^{(n)}Z_t^{(n)}|\mathcal{Y}_t]}{E[Z_t^{(n)}|\mathcal{Y}_t]}-\frac{E[X_u^{(n)}Z_t^{(n)}|\mathcal{Y}_t]\otimes E[X_s^{(n)}Z_t^{(n)}|\mathcal{Y}_t]}{E[Z_t^{(n)}|\mathcal{Y}_t]^2} \right\}.
  \end{align*}
  According to (\ref{eq3-24}), this is deterministic. Thus, by the Clark-Ocone formula \citep{nualart2006malliavin}, we can write
  \begin{align*}
    &\frac{E_Q[X_s^{(n)}Z_t^{(n)}|\mathcal{Y}_t]}{E_Q[Z_t^{(n)}|\mathcal{Y}_t]}
    =g^{(n)}(s;t)+\int_0^t\left( \frac{E[X_u^{(n)}\otimes X_s^{(n)}Z_t^{(n)}|\mathcal{Y}_t]}{E[Z_t^{(n)}|\mathcal{Y}_t]}\right.\\
    &\left.\qquad -\frac{E_Q[X_u^{(n)}Z_t^{(n)}|\mathcal{Y}_t]\otimes E_Q[X_s^{(n)}Z_t^{(n)}|\mathcal{Y}_t]}{E_Q[Z_t^{(n)}|\mathcal{Y}_t]^2} \right)^\top c(u)^\top (\sigma(u)\sigma(u)^\top)^{-1}dY_u,
  \end{align*}
  where $g^{(n)}(s;t)$ is a deterministic function. Due to Lemma \ref{lemma-3-2}, for every $s,t$, $g^{(n)}(s;t)$ converges, and
  \begin{align*}
    \left\{\frac{E_Q[X_u^{(n)}\otimes X_s^{(n)}Z_t^{(n)}|\mathcal{Y}_t]}{E_Q[Z_t^{(n)}|\mathcal{Y}_t]}
    -\frac{E_Q[X_u^{(n)}Z_t^{(n)}|\mathcal{Y}_t]\otimes E_Q[X_s^{(n)}Z_t^{(n)}|\mathcal{Y}_t]}{E_Q[Z_t^{(n)}|\mathcal{Y}_t]^2}\right\}_{0\leq u \leq t}
  \end{align*}
  converges in $L^2(\Omega;H)$ where $H=L^2([0,t])$. On the other hand, we have shown above that for each $s,u \in [0,t]$, this converges to $\gamma(s,u;t)$ in $L^2$. Therefore, by letting $n\to \infty$, we obtain
  \begin{align*}
    E[X_s|\mathcal{Y}_t]&=\frac{E_Q[X_sZ_t|\mathcal{Y}_t]}{E_Q[Z_t|\mathcal{Y}_t]} =g(s;t)+\int_0^t \gamma(s,u;t)c(u)^\top(\sigma(u)\sigma(u)^\top)^{-1}dY_u,
  \end{align*}
  where $g(s;t)$ is a deterministic function. By taking expectations of both sides under $P$, it follows that 
  \begin{align*}
    E[X_s]&=E[E[X_s|\mathcal{Y}_t]]
    =g(s;t)+\int_0^t \gamma(s,u;t)c(u)^\top(\sigma(u)\sigma(u)^\top)^{-1}c(u)E[X_u]du.
  \end{align*}
  This gives the expression of $g(s;t)$, which leads to the desired result.
\end{proof}

Now, the conditional density of the path $\{X_s\}_{0 \leq s \leq t}$ given $\mathcal{Y}_t$ is reduced to the computation of $\gamma(s, u; t)$ in (\ref{eq-gamma}). Since this is a ratio of unconditional expectations, we can apply well-established techniques from stochastic analysis.

To compute $\gamma$, for each fixed \( t \geq 0 \), we introduce an \( M_{d_1}(\mathbb{R}) \)-valued differentiable function \( \phi(s; t) \), defined as the negative semidefinite symmetric solution to the matrix Riccati equation
\begin{align}
  \label{eq-def-phi-tmp}
  \begin{split}
    \frac{d\phi}{ds}(s; t) &= -\phi(s; t) b(s) b(s)^\top \phi(s; t) - a(s)^\top \phi(s; t) - \phi(s; t) a(s) \\
    &+ c(s)^\top (\sigma(s) \sigma(s)^\top)^{-1} c(s)
  \end{split}
\end{align}
with the boundary condition \( \phi(t; t) = 0 \). This solution is well-defined. Specifically, the Riccati equation
\begin{align*}
  \frac{d\tilde{\phi}}{ds}(s; t) &= -\tilde{\phi}(s; t) b(t - s) b(t - s)^\top \tilde{\phi}(s; t)  - a(t - s)^\top \tilde{\phi}(s; t) - \tilde{\phi}(s; t) a(t - s) \\
  &\quad + c(t - s)^\top (\sigma(t - s) \sigma(t - s)^\top)^{-1} c(t - s)
\end{align*}
with \( \tilde{\phi}(0; t) = 0 \) has a positive-semidefinite solution \( \tilde{\phi}(s; t) \) for \( 0 \leq s \leq t \), according to Theorems 2.1 and 2.2 in \citet{potter1965matrix}. It then follows immediately that \( \phi(s; t) = -\tilde{\phi}(t - s; t) \) satisfies (\ref{eq-def-phi}), and uniqueness can be established using Gronwall's lemma.

Also, let \( \{\tilde{V}_s\} \) be a \( d_1 \)-dimensional Brownian motion on \( (\Omega, \mathcal{F}, \{\mathcal{F}_t\}, P) \) that is independent of \( \{Y_s\} \) and \( X_0 \). Let \( \{\xi_{s; t}^0\}_{0 \leq s \leq t} \) be the solution to the stochastic differential equation
\begin{align}
  \label{eq-def-xi_0}
  d_s\xi_{s; t}^0 = \left\{ a(s) + b(s) b(s)^\top \phi(s; t) \right\} \xi_{s; t}^0 \, ds + b(s) \, d\tilde{V}_s,~~\xi_{0;t}^0=X_0.
\end{align}

We then obtain the following key lemma, which follows as a consequence of Girsanov's theorem and It\^o's formula.
\begin{lemma}
  Let $X.$ and $\xi_{\bf{\cdot}}^0$ be the paths of $\{X_s\}_{0\leq s\leq t}$ and $\{\xi_{s,t}^0\}_{0\leq s\leq t}$, and $\mu_X$ and $\mu_{\xi^0}$ be their distributions on $P$. 
  \begin{align}
    \label{eq-measure-change}\begin{split}      
    \frac{d\mu_{\xi^0}}{d\mu_X}(X.)=\exp\biggl( &-\frac{1}{2}\int_0^t X_s^\top c(s)^\top (\sigma(s)\sigma(s)^\top)^{-1} c(s) X_sds\\
    &-\frac{1}{2}\int_0^t b(s)^\top \phi(s;t)b(s)ds-\frac{1}{2}X_0^\top \phi(0;t) X_0\biggr).
    \end{split}
  \end{align}
\end{lemma}
\begin{proof}
  Due to (7.138) in \citet{Liptser2001}, $\mu_X$ and $\mu_{\xi^0}$ are equivalent, and it holds that 
  \begin{align}
    \label{eq3-3}\begin{split}      
    \frac{d\mu_{\xi^0}}{d\mu_X}(X.)=\exp\biggl(&\int_0^t X_{s}^\top\phi(s;t) dX_s\\
    &-\frac{1}{2}\int_0^tX_s^\top \phi(s;t) (2a(s)+b(s)b(s)^\top \phi(s;t))X_s ds \biggr).
    \end{split}
  \end{align}
  On the other hand, by It\^o's formula and (\ref{eq-def-phi-tmp}), we have
  \begin{align*}
    &X_t^\top \phi(t;t) X_t-X_0^\top \phi(0;t) X_0\\
    =&2\int_0^t X_s^\top\phi(s;t)dX_s+\int_0^t X_s^\top \frac{\partial}{\partial s}\phi(s;t) X_sds + \int_0^t b(s)^\top \phi(s;t)b(s)ds\\
    =&2\int_0^t X_s^\top\phi(s;t)dX_s-\int_0^t X_s^\top \phi(s;t)b(s)b(s)^\top \phi(s;t) X_sds \\
    &-\int_0^t X_s^\top a(s)^\top \phi(s;t) X_sds-\int_0^t X_s^\top  \phi(s;t)a(s) X_sds\\
    &+\int_0^t X_s^\top c(s)^\top (\sigma(s)\sigma(s)^\top)^{-1} c(s) X_sds
    + \int_0^t b(s)^\top \phi(s;t)b(s)ds\\
    =&2\int_0^t X_s^\top\phi(s;t)dX_s-\int_0^t X_s^\top \phi(s;t)b(s)b(s)^\top \phi(s;t) X_sds \\
    &-2\int_0^t X_s^\top  \phi(s;t)a(s) X_sds\\
    &+\int_0^t X_s^\top c(s)^\top (\sigma(s)\sigma(s)^\top)^{-1} c(s) X_sds
    + \int_0^t b(s)^\top \phi(s;t)b(s)ds.
  \end{align*}
  Therefore, noting that $\phi(t;t)=0$, it holds
  \begin{align*}
    &\int_0^t X_s^\top\phi(s;t)dX_s\\
    =&\frac{1}{2}\int_0^t X_s^\top \phi(s;t)b(s)b(s)^\top \phi(s;t) X_sds+\int_0^t X_s^\top  \phi(s;t)a(s) X_sds\\
    &-\frac{1}{2}\int_0^t X_s^\top c(s)^\top (\sigma(s)\sigma(s)^\top)^{-1} c(s) X_sds\\
    &-\frac{1}{2}\int_0^t b(s)^\top \phi(s;t)b(s)ds-\frac{1}{2}X_0^\top \phi(0;t) X_0.
  \end{align*}  
  Together with (\ref{eq3-3}), we obtain the desired result.  
\end{proof}

According to this result, the exponential term in the expression of $\gamma(s,u;t)$ in (\ref{eq-gamma}) can be rewritten as
\begin{align}
  \label{eq-rewrite-exp}
  \begin{split}
    &\exp\left( -\frac{1}{2}\int_0^t X_s^\top c(s)^\top (\sigma(s)\sigma(s)^\top)^{-1} c(s) X_s \, ds \right) \\
    = &\exp\left( \frac{1}{2}X_0^\top \phi(0;t) X_0+\frac{1}{2}\int_0^t b(s)^\top \phi(s;t)b(s)ds \right) \frac{d\mu_{\xi^0}}{d\mu_X}(X_\cdot).
  \end{split}  
\end{align}
Now, let us modify the initial value of $\xi_{s;t}^0$ and define a new process $\xi_{s;t}$ by
\begin{align}
  \label{eq-def-xi-tmp}d_s \xi_{s;t} = \left\{ a(s) + b(s) b(s)^\top \phi(s;t) \right\} \xi_{s;t} \, ds + b(s) \, d\tilde{V}_s,
\end{align}
where \( \xi_{0;t} \) is a Gaussian random variable, independent of \( \{\tilde{V}_s\} \) and \( \{Y_s\} \), with mean zero and covariance matrix
\begin{align}
  \label{eq-cov-xi0-tmp}
  \mathrm{Var}[\xi_{0;t}] = V[X_0]^{1/2} \left( I_{d_1} - V[X_0]^{1/2} \phi(0;t) V[X_0]^{1/2} \right)^{-1} V[X_0]^{1/2}.
\end{align}
Here, note that the matrix \( I_{d_1} - V[X_0]^{1/2} \phi(0;t) V[X_0]^{1/2} \) is positive definite since \( \phi \) is defined as a negative semidefinite solution. Then, \( \gamma(s,u;t) \) coincides with the covariance matrix of the process \( \xi \).
\begin{proposition}
  For every \( s, u \in [0, t] \), we have \( \gamma(s,u;t) = \mathrm{Cov}(\xi_{s;t}, \xi_{u;t}) \).
\end{proposition}
\begin{proof}
 Using (\ref{eq-rewrite-exp}), we obtain 
  \begin{align}
    &\gamma(s,u;t)
    =\frac{\displaystyle E\left[X_s\otimes X_u\exp\left( \frac{1}{2}X_0^\top \phi(0;t)X_0 \right)\frac{d\mu_{\xi^0}}{d\mu_X}(X.)\middle|\mathcal{Y}_t\right]}{\displaystyle E\left[\exp\left( \frac{1}{2}X_0^\top \phi(0;t)X_0 \right)\frac{d\mu_{\xi^0}}{d\mu_X}(X.)\middle|\mathcal{Y}_t\right]}\nonumber\\
    &-\frac{\displaystyle E\left[X_s\exp\left( \frac{1}{2}X_0^\top \phi(0;t)X_0 \right)\frac{d\mu_{\xi^0}}{d\mu_X}(X.)\middle|\mathcal{Y}_t\right]}{\displaystyle E\left[\exp\left( \frac{1}{2}X_0^\top \phi(0;t)X_0 \right)\frac{d\mu_{\xi^0}}{d\mu_X}(X.)\middle|\mathcal{Y}_t\right]}\nonumber\\
    &\otimes \frac{\displaystyle E\left[\exp\left( \frac{1}{2}X_0^\top \phi(0;t)X_0 \right)\frac{d\mu_{\xi^0}}{d\mu_X}(X.)\middle|\mathcal{Y}_t\right]}{\displaystyle E\left[\exp\left( \frac{1}{2}X_0^\top \phi(0;t)X_0 \right)\frac{d\mu_{\xi^0}}{d\mu_X}(X.)\middle|\mathcal{Y}_t\right]}\nonumber\\
    \label{eq-cov-breve-xi}\begin{split}      
    =&\frac{\displaystyle E_Q\left[\xi^0_{s;t}\otimes \xi^0_{u;t}\exp\left( \frac{1}{2}X_0^\top \phi(0;t)X_0 \right)\right]}{\displaystyle E_Q\left[\exp\left( \frac{1}{2}X_0^\top \phi(0;t)X_0 \right)\right]}\\
    &-\frac{\displaystyle E_Q\left[\xi^0_{s;t}\exp\left( \frac{1}{2}X_0^\top \phi(0;t)X_0 \right)\right]}{\displaystyle E_Q\left[\exp\left( \frac{1}{2}X_0^\top \phi(0;t)X_0 \right)\right]}
    \otimes \frac{\displaystyle E_Q\left[\xi^0_{u;t}\exp\left( \frac{1}{2}X_0^\top \phi(0;t)X_0 \right)\right]}{\displaystyle E_Q\left[\exp\left( \frac{1}{2}X_0^\top \phi(0;t)X_0 \right)\right]}.
    \end{split}
  \end{align}
  Now, write $\xi_{s;t}^0$ as 
  \begin{align}
    \label{eq-decompisition-xi0}\xi^0_{s; t} = \alpha(s; t) X_0 + \overline{\xi}_{s; t},
  \end{align}
  where
  \begin{align*}
    &\alpha(s;t)=\exp\left( \int_s^t \left\{ a(r)+b(r)b(r)^\top \phi(r;t) \right\} \right),~~\\
    &\overline{\xi}_{s; t}=\int_{0}^s\exp\left( \int_u^s \left\{ a(r)+b(r)b(r)^\top \phi(r;t) \right\} \right)b(s)d\tilde{V}_s
  \end{align*}
  Using this expression, and due to $\xi_{s;t}=\alpha(s;t)\xi_{0;t}+\bar{\xi}_{s;t}$ and (\ref{eq-cov-xi0-tmp}), (\ref{eq-cov-breve-xi}) becomes equivalent to  
  \begin{align*}
    &\mathrm{Cov}(\overline{\xi}_{s;t},\overline{\xi}_{u;t})
    +\frac{\displaystyle E\left[\alpha(s;t)X_0\otimes \alpha(u;t)X_0\exp\left( \frac{1}{2}X_0^\top \phi(0;t)X_0 \right)\right]}{\displaystyle E\left[\exp\left( \frac{1}{2}X_0^\top \phi(0;t)X_0 \right)\right]}\\
    &-\frac{\displaystyle E\left[\alpha(s;t)X_0\exp\left( \frac{1}{2}X_0^\top \phi(0;t)X_0 \right)\right]}{\displaystyle E\left[\exp\left( \frac{1}{2}X_0^\top \phi(0;t)X_0 \right)\right]}\\
    &\otimes \frac{\displaystyle E\left[\alpha(u;t)X_0\exp\left( \frac{1}{2}X_0^\top \phi(0;t)X_0 \right)\right]}{\displaystyle E\left[\exp\left( \frac{1}{2}X_0^\top \phi(0;t)X_0 \right)\right]}\\
    =&\mathrm{Cov}(\overline{\xi}_{s;t},\overline{\xi}_{u;t})+\alpha(s;t)V[X_0]^\frac{1}{2}(I-V[X_0]^\frac{1}{2}\phi(0;t)V[X_0]^\frac{1}{2})^{-1}V[X_0]^\frac{1}{2}\alpha(u;t)^\top\\
    =&\mathrm{Cov}(\xi_{s;t},\xi_{u;t}).
  \end{align*}

\end{proof}

To summarize the results, we obtain the following theorem.
\begin{theorem}\label{main-theorem}
  The conditional distribution of $\{X_s\}_{0\leq s\leq t}$ given $\mathcal{Y}_t$ is a Gaussian process with
  \begin{align}
    \label{eq-def-mu}\begin{split}
      &E[X_s|\mathcal{Y}_t]=E[X_s] \\
    &+ \int_0^t \mathrm{Cov}(\xi_{s; t}, \xi_{u; t}) c(u)^\top (\sigma(u) \sigma(u)^\top)^{-1} (dY_u - c(u) E[X_u] \, du),
    \end{split}\\
    \label{eq-def-gamma}&\mathrm{Cov}(X_s,X_u|\mathcal{Y}_t)=\mathrm{Cov}(\xi_{s;t},\xi_{u;t}),
  \end{align}
  where
\begin{align}
  \label{eq-def-phi}
  \begin{split}
    \frac{d\phi}{ds}(s; t) &= -\phi(s; t) b(s) b(s)^\top \phi(s; t) - a(s)^\top \phi(s; t) - \phi(s; t) a(s)\\ 
    &+ c(s)^\top (\sigma(s) \sigma(s)^\top)^{-1} c(s),
  \end{split}
\end{align}
with $\phi(t;t)=0$, and
\begin{align}
  \label{eq-def-xi}
  d_s\xi_{s; t} = \left\{ a(s) + b(s) b(s)^\top \phi(s; t) \right\} \xi_{s; t} \, ds + b(s) \, d\tilde{V}_s
\end{align}
with $E[\xi_{0; t}]=0$ and
\begin{align}
  \label{eq-cov-xi0}
  V[\xi_{0; t}] = V[X_0]^{\frac{1}{2}} \left( I_{d_1} - V[X_0]^{\frac{1}{2}} \phi(0; t) V[X_0]^{\frac{1}{2}} \right)^{-1} V[X_0]^{\frac{1}{2}}.
\end{align}
Therefore, the smoothing error $\{X_s-E[X_s|\mathcal{Y}_t]\}_{0\leq s\leq t}$ has the same distribution as the Ornstein-Uhlenbeck process $\{\xi_{s;t}\}_{0\leq s\leq t}$.
\end{theorem}
This theorem enables pathwise sampling of the conditional error. Its practical applications are suggested in Section \ref{section-application}. 

\section{Filtering and smoothing equations}\label{section-linear-equations}
In this section, we show that filtering and smoothing equations can be derived as direct consequences of the main theorem. We denote
\begin{align*}
  \mu_{s;t} = E[X_s \mid \mathcal{Y}_t], \quad \gamma(s,u;t) = \mathrm{Cov}(X_s, X_u \mid \mathcal{Y}_t),
\end{align*}
We begin by deriving the Bryson–Frazier smoother as an immediate corollary of Theorem~\ref{main-theorem}.

\begin{proposition} {\bf (Bryson–Frazier smoother)}\\
  The conditional mean $\mu_{s;t}$ satisfies the equation 
  \begin{align*}
    d_s\mu_{s;t}=&\left\{ a(s) + b(s) b(s)^\top \phi(s; t) \right\}\mu_{s;t}ds-b(s) b(s)^\top \phi(s; t)E[X_s]ds\\
    &+b(s)b(s)^\top\rho_{s;t} ds,\\
    \mu_{0;t}=& V[X_0]^{\frac{1}{2}} \left( I_{d_1} - V[X_0]^{\frac{1}{2}} \phi(0; t) V[X_0]^{\frac{1}{2}} \right)^{-1} V[X_0]^{\frac{1}{2}}\rho_{0;t}&
  \end{align*}
  where $\rho_{s;t}$ is the solution of the equation
  \begin{align*}
    d_s\rho_{s;t}=&\left\{ a(s) + b(s) b(s)^\top \phi(s; t) \right\}\rho_{s;t}ds\\
    &+c(s)^\top(\sigma(s)\sigma(s))^{-1}(dY_s - c(s)E[X_s]ds),\\
    \rho_{t;t}=&0.
  \end{align*}
\end{proposition}
\begin{proof}
  The solution of (\ref{eq-def-xi}) is represented as
  \begin{align*}
    \xi_{s;t}=\alpha(0,s;t)\xi_{0;t}+\int_0^s \alpha(r,s;t)b(r)d\tilde{V}_r,
  \end{align*}
  where $\alpha(u,s;t)$ is the solution of 
  \begin{align}
    \label{def-alpha-u-s-t}\frac{\partial}{\partial s}\alpha(u,s;t)=\left\{ a(s) + b(s) b(s)^\top \phi(s; t) \right\}\alpha(u,s;t),~~\alpha(u,u;t)=I.
  \end{align}
  Thus we have
  \begin{align}
    \label{eq-representation-gamma}\begin{split}
      \gamma(s,u;t)=E[\xi_{s;t}\xi_{u;t}^\top]=&\alpha(0,s;t)V[\xi_{0;t}]\alpha(0,u;t)^\top\\
    &+\int_0^{s\wedge u}\alpha(r,s;t)b(r)b(r)^\top \alpha(r,u;t)^\top dr
    \end{split}
  \end{align}
  and
  \begin{align*}
    \frac{\partial}{\partial s}\gamma(s,u;t)=\begin{cases}
      \left\{ a(s) + b(s) b(s)^\top \phi(s; t) \right\}\gamma(s,u;t)&(s<u)\\
      \left\{ a(s) + b(s) b(s)^\top \phi(s; t) \right\}\gamma(s,u;t)+b(s)b(s)^\top \alpha(s,u;t)^\top &(s\geq u)
    \end{cases}.
  \end{align*}
  Therefore, differentiating (\ref{eq-def-mu}) yields
  \begin{align*}
    &d_s\mu_{s;t}=\left\{ a(s) + b(s) b(s)^\top \phi(s; t) \right\}\mu_{s;t}ds-b(s) b(s)^\top \phi(s; t)E[X_s]ds\\
    &+b(s)b(s)^\top\int_s^t  \alpha(s,u;t)^\top c(u)^\top (\sigma(u) \sigma(u)^\top)^{-1} (dY_u-c(u)E[X_u]du) ds.
  \end{align*}
  Now we obtain the desired result by writing 
  \begin{align*}
    \rho_{s;t}=\int_s^t  \alpha(s,u;t)^\top c(u)^\top (\sigma(u) \sigma(u)^\top)^{-1} (dY_u-c(u)E[X_u]du).
  \end{align*}
\end{proof}

Next, we derive the differential equation for \( \gamma(t,t;t)=V[X_t|\mathcal{Y}_t] \). To this end, we define \( \gamma(t) \) as the positive-semidefinite solution of the Riccati equation
\begin{align}
  \label{eq-gamma}
  \begin{split}
    \frac{d}{dt} \gamma(t) &= -\gamma(t) c(t)^\top (\sigma(t) \sigma(t)^\top)^{-1} c(t) \gamma(t) + a(t) \gamma(t) + \gamma(t) a(t)^\top + b(t) b(t)^\top
  \end{split}
\end{align}
with \( \gamma(0) = V[X_0] \). The existence of the solution is guaranteed by Theorems 2.1 and 2.2 in \citet{potter1965matrix}.

To avoid potential singularity issues with \( \gamma(t) \), we introduce \( \gamma^\epsilon(t) \) and \( \xi_{s; t}^\epsilon \) as solutions to the same equations
\begin{align}
  \label{eq-def-gamma-epsilon-1}
  \begin{split}    
    \frac{d}{dt} \gamma^\epsilon(t) &= -\gamma^\epsilon(t) c(t)^\top (\sigma(t) \sigma(t)^\top)^{-1} c(t) \gamma^\epsilon(t) + a(t) \gamma^\epsilon(t) + \gamma^\epsilon(t) a(t)^\top \\
    &+ b(t) b(t)^\top
  \end{split}
\end{align}
and
\begin{align}
  \label{eq-def-xi-epsilon-1}
  d \xi_{s; t}^\epsilon = \left\{ a(s) + b(s) b(s)^\top \phi(s; t) \right\} \xi_{s; t}^\epsilon \, ds + b(s) \, d\tilde{V}_s,
\end{align}
where $\gamma^\epsilon(0) = V[X_0] + \epsilon I_{d_1}$ and \( \xi_{0; t}^\epsilon \) is a Gaussian random variable with mean zero and covariance
\begin{align}
  \label{eq-def-xi-epsilon-2}
  V[\xi_{0; t}^\epsilon] = \left\{(V[X_0] + \epsilon I_{d_1})^{-1} - \phi(0; t) \right\}^{-1}.
\end{align}
By Corollary 1 to Theorem 2.1 in \citet{potter1965matrix}, \( \gamma^\epsilon(t) \) is strictly positive.

\begin{lemma}\label{lemma-gamma-epsilon}
  (1) For every $t\geq0$, $\gamma^\epsilon(s)$ is bounded for $\epsilon \in [0,1]$ and $s \in [0,t]$.\\
  (2) The matrix $\gamma^\epsilon(t)$ converges uniformly to $\gamma(t)$ on any compact set as $\epsilon \to +0$.
\end{lemma}
\begin{proof}
  (1) Take an arbitrary $x \in \mathbb{R}^{d_1}$. Then (\ref{eq-gamma}) yields
  \begin{align*}
    \frac{d}{dt}x^\top \gamma^\epsilon(t) x=&-x^\top\gamma^\epsilon(t)c(t)^\top(\sigma(t)\sigma(t)^\top)^{-1}c(t)\gamma^\epsilon(t)x+2x^\top a(t) \gamma^\epsilon(t)x\\
    &+x^\top b(t)b(t)^\top x\\
    \leq & 2|x^\top a(t) \gamma^\epsilon(t)^\frac{1}{2} \gamma^\epsilon(t)^\frac{1}{2}x|+x^\top b(t)b(t)^\top x\\
    \leq &2\sqrt{x^\top a(t) \gamma^\epsilon(t)a(t)^\top  x}\sqrt{x^\top  \gamma^\epsilon(t) x}+x^\top b(t)b(t)^\top x\\
    \leq &2\|a(t)\|_2 \|\gamma^\epsilon(t)\|_2|x|^2+\|b(t)\|_2^2 |x|^2,
  \end{align*}
  where $\|\cdot\|_2$ is the matrix 2-norm, and we used
  \begin{align*}
    x^\top a(t) \gamma^\epsilon(t)a(t)^\top x&=|\gamma^\epsilon(t)^\frac{1}{2}a(t)^\top x|^2\leq \|\gamma^\epsilon(t)\|_2\|a(t)\|_2^2|x|^2.
  \end{align*}
  From this, it follows that
  \begin{align*}
    \|\gamma^\epsilon(t)\|_2=\max_{|x|=1}x^\top \gamma^\epsilon(t) x\leq \|\gamma^\epsilon(0)\|_2+\int_0^t \{2\|a(s)\|_2\|\gamma^\epsilon(s)\|_2+\|b(s)\|_2^2\}ds.
  \end{align*}
  Therefore, by the Gronwall's lemma, we obtain
  \begin{align*}
    \|\gamma^\epsilon(t)\|_2\leq & \|\gamma^\epsilon(0)\|_2\exp\left( 2\int_0^t\|a(r)\|_2dr \right)\\
    &+\int_0^t \exp\left( 2\int_s^t\|a(r)\|_2dr \right)\|b(s)\|_2^2ds,
  \end{align*}
  which leads to the desired result.\\
  Statement (2) can be proven routinely using the result of (1) along with Gronwall's lemma.
\end{proof}

\begin{proposition}\label{prop-gamma}
  It holds that $\gamma(t)=\gamma(t,t;t).$
\end{proposition}
\begin{proof}
  First, assume that $\gamma(t)$ is strictly positive for every $t\geq 0$. Then we have 
  \begin{align*}
    \frac{d}{ds}\left( \gamma(s)^{-1} \right)
    =&-\gamma(s)^{-1} \frac{d}{ds}\gamma(s) \gamma(s)^{-1}\\
    =&-\gamma(s)^{-1}b(s)b(s)^\top \gamma(s)^{-1}-a(s)^\top \gamma(s)^{-1}-\gamma(s)^{-1}a(s)\\
    &+c(s)^\top (\sigma(s)\sigma(s)^\top)^{-1}c(s).
  \end{align*}
  This means that $\gamma(s)^{-1}$ satisfies the same equation (\ref{eq-def-phi}) as $\phi(s;t)$. Hence, if we set $z(s;t)=\gamma(s)^{-1}-\phi(s;t)$, then it follows that
  \begin{align*}
    &\frac{d}{ds}z(s;t)
    =-\left\{ \gamma(s)^{-1}b(s)b(s)^\top \gamma(s)^{-1}
    -\phi(s;t)^{-1}b(s)b(s)^\top \phi(s;t)^{-1} \right\}\\
    &-a(s)^\top \left\{ \gamma(s)^{-1}-\phi(s;t) \right\} -\left\{ \gamma(s)^{-1}-\phi(s;t) \right\}a(s)\\
    =&-z(s;t)b(s)b(s)^\top z(s;t)-\left\{ a(s)^\top+\phi(s;t)b(s)b(s)^\top \right\}z(s;t)\\
    &-z(s;t)\left\{ a(s)+b(s)b(s)^\top \phi(s;t) \right\}.
  \end{align*}
  Furthermore, if we set $w(s;t)=z(s;t)^{-1}$ (this is well-defined since $\phi(s;t)$ is negative-semidefinite), we have
  \begin{align}
    \label{eq-derivative-w}\begin{split}
      \frac{d}{ds}w(s;t)=&-z(s;t)^{-1} \frac{d}{ds}z(s;t) z(s;t)^{-1}\\
    =&b(s)b(s)^\top+w(s;t)\left\{ a(s)^\top+\phi(s;t)b(s)b(s)^\top \right\}\\
    &+\left\{ a(s)+b(s)b(s)^\top\phi(s;t) \right\}w(s;t)
    \end{split}
  \end{align}
  Therefore, $w(s;t)$ can be written as
  \begin{align}
    \label{eq-expression-w}\begin{split}
      w(s;t)=&\alpha(0,s;t)V[\xi_{0;t}] \alpha(0,u;t)+\int_0^s \alpha(u,s;t)b(u) b(u)^\top\alpha(u,s;t)^\top du.
    \end{split}   
  \end{align}
  where $\alpha$ is defined in (\ref{def-alpha-u-s-t}). Here, we used $\gamma(0)=V[X_0]$ and (\ref{eq-cov-xi0}) to establish $w(0;s)=(V[X_0]^{-1}-\phi(0;t))^{-1}=V[\xi_{0;t}].$

  Hence, it follows immediately from (\ref{eq-representation-gamma}) and (\ref{eq-expression-w}) that \(\gamma(s,s;t)= V[\xi_{s; t}] = w(s; t) \). In particular, since
  \begin{align}
    \label{eq-w-inverse-formula}
    w(s; t) = \{\gamma(s)^{-1} - \phi(s; t)\}^{-1}
  \end{align}
  and \( \phi(t; t) = 0 \), we have \( w(t; t) = \gamma(t) \), which gives the desired result.
  
  In the case where \( \gamma(t) \) becomes singular for some \( t \), we replace \( V[X_0] \) with \( V[X_0] + \epsilon I_{d_1} \) for \( \epsilon > 0 \), yielding $\gamma^\epsilon(t) = V[\xi_{t; t}^\epsilon]$, since \( \gamma^\epsilon(t) \) is strictly positive. Then let \( \epsilon \to 0 \) to obtain the conclusion.  
\end{proof}
\begin{remark}\label{remark-inverse}
  The proof reveals that $\psi(s;t)$ can be expressed as \begin{align*}
    \psi(s;t) = \left\{\gamma(s)^{-1} - \gamma(s,s;t)^{-1}\right\}^{-1},
  \end{align*}
  provided that $\gamma(s)$, $\gamma(s,s;t)$ and $\gamma(s)^{-1} - \gamma(s,s;t)^{-1}$ are non-singular. This observation explains why the inverses $\gamma(s)^{-1}$ and $\gamma(s,s;t)^{-1}$ inevitably appear when one attempts to derive the BF smoother within the frameworks of the Kalman–Bucy filter and the RTS smoother.
\end{remark}

\begin{proposition}\label{prop-expression-cov-xi}
  Let $0\leq u \leq s\leq t$. If $\gamma(r)$ is non-singular for every $r\in [0,t]$, it holds that 
  \begin{align*}
    d_u\gamma(s,u;t)=\gamma(s,u;t)\left\{ a(u)^\top+\gamma(u)^{-1}b(u)b(u)^\top \right\}du.
  \end{align*}
\end{proposition}
\begin{proof}
  Due to (\ref{eq-representation-gamma}) and $w(s;t)=\gamma(s,s;t)$, we can express
  \begin{align}
    \gamma(s,u;t)\label{eq-cov-xi-2}=&\alpha(u,s;t)w(u;t),
  \end{align} 
  where $w(u;t)=V[\xi_{u;t}]$ is given by (\ref{eq-expression-w}). Thus, it follows from (\ref{eq-derivative-w}) and (\ref{eq-w-inverse-formula}) that
  \begin{align*}
    &d_u\gamma(s,u;t)\\
    =&\alpha(u,s;t) \left[ -\left\{ a(u)+b(u)b(u)^\top \phi(u;t)\right\}w(u;t)+\frac{\partial}{\partial u}w(u;t) \right]du\\
    =&\alpha(u,s;t)\\
    &\times \left[ -\left\{ a(u)+b(u)b(u)^\top \phi(u;t)\right\}w(u;t)+b(u)b(u)^\top\right.\\
    &+w(u;t)\left\{ a(u)^\top+\phi(u;t)b(u)b(u)^\top \right\}\\
    &+\left.\left\{ a(u)+b(u)b(u)^\top\phi(u;t) \right\}w(u;t)\right]du\\
    =&\alpha(u,s;t)w(u,t)\\
    &\times \left\{ w(u;t)^{-1}b(u)b(u)^\top+a(u)^\top+\phi(u;t)b(u)b(u)^\top \right\}du\\
    =&\mathrm{Cov}(\xi_{s;t},\xi_{u;t}) \left[ \left\{ \gamma(u)^{-1}-\phi(u;t) \right\}b(u)b(u)^\top+a(u)^\top+\phi(u;t)b(u)b(u)^\top \right]du\\
    =&\gamma(s,u;t)\left\{ a(u)^\top+\gamma(u)^{-1}b(u)b(u)^\top \right\}du,
  \end{align*}
  which implies the conclusion. 
\end{proof}

\begin{proposition}\label{prop-equation-xi-t-s}
  For every $s\geq0$, it holds that 
  \begin{align*}
    d_t\gamma(t,s;t)=\left\{ a(t)-\gamma(t)c(t)^\top(\sigma(t)\sigma(t)^\top)^{-1}c(t)\right\}\gamma(t,s;t)dt.
  \end{align*}
\end{proposition}
\begin{proof}
  First, assume that $\gamma(t)$ is non-singular for every $t\geq 0$. In this case, by Propositions \ref{prop-gamma} and \ref{prop-expression-cov-xi}, we can write
  \begin{align*}
    \gamma(t,s;t)=\mathrm{Cov}(\xi_{t;t},\xi_{s;t})=\gamma(t)\beta(s,t)^\top.
  \end{align*}
  where $\beta(u,s)$ is the solution of
  \begin{align}
    \label{eq-def-beta}\frac{\partial}{\partial s}\beta(u,s)=\left\{ a(r)+b(r)b(r)^\top\gamma(r)^{-1} \right\}\beta(u,s),~~\beta(u,u)=I.
  \end{align}
  Hence, it holds
  \begin{align*}
    &d_t\gamma(t,s;t)\\
    =&\left[ \frac{d}{dt}\gamma (t)-\gamma(t)\{a(t)^\top+\gamma(t)^{-1}b(t)b(t)^\top\}\right]\beta(s,t)^\top dt\\
    =&\left\{ -\gamma(t)c(t)^\top(\sigma(t)\sigma(t)^\top)^{-1}c(t)\gamma(t)+a(t)\gamma(t)\right\} \beta(s,t)^\top dt\\
    =&\left\{ -\gamma(t)c(t)^\top(\sigma(t)\sigma(t)^\top)^{-1}c(t)+a(t)\right\}\gamma(t,s;t)dt,
  \end{align*}
  which implies the conclusion. For the case where \( \gamma(t) \) is singular, we have
  \begin{align*}
    d_t\mathrm{Cov}(\xi_{t;t}^\epsilon,\xi_{s;t}^\epsilon)
    =\left\{ -\gamma^\epsilon(t)c(t)^\top(\sigma(t)\sigma(t)^\top)^{-1}c(t)+a(t)\right\}\mathrm{Cov}(\xi_{t;t}^\epsilon,\xi_{s;t}^\epsilon)dt
  \end{align*}
  and the uniform convergence of \( \gamma^\epsilon \) on any compact set yields the desired result.
\end{proof}

The previous proposition provides the well-known results of Kalman-Bucy filter and Rauch-Tung-Striebel smoother.

\begin{proposition}{\bf (Kalman-Bucy filter)}\\
  The process $\{\mu_{t,t}\}_{t\geq 0}$ satisfies the stochastic differential equation
  \begin{align*}
    d\mu_{t,t}=a(t)\mu_{t,t}dt+\gamma(t) c(t)^\top(\sigma(t)\sigma(t)^\top)^{-1}\{dY_t-c(t)\mu_{t,t}dt\}.
  \end{align*} 
\end{proposition}
\begin{proof}
  Set
  \begin{align}
    \label{eq3-12}\tilde{\mu}_{t;t}=\mu_{t;t}-E[X_t],~~\tilde{Y}_t=Y_t-\int_0^t c(u)E[X_u]du.
  \end{align}
  Then according to Proposition \ref{prop-equation-xi-t-s} and (\ref{eq-def-mu}), we have
  \begin{align*}
    \tilde{\mu}_{t,t}=&\int_0^t\gamma(t,u;t)c(u)^\top(\sigma(u)\sigma(u)^\top)^{-1}d\tilde{Y}_u\\
    =&\int_0^t\left[\gamma(u)+
    \int_u^t\left\{ -\gamma(s)c(s)^\top(\sigma(s)\sigma(s)^\top)^{-1}c(s)+a(s)\right\}\gamma(s,u;s)ds\right]\\
    &\times c(u)^\top(\sigma(u)\sigma(u)^\top)^{-1}d\tilde{Y}_u\\
    =&\int_0^t a(s)\int_0^s \gamma(s,u;s)c(u)^\top(\sigma(u)\sigma(u)^\top)^{-1}d\tilde{Y}_uds\\
    &+\int_0^t\gamma(u) c(u)^\top(\sigma(u)\sigma(u)^\top)^{-1}d\tilde{Y}_u\\
    &-\int_0^t\gamma(s)c(s)^\top(\sigma(s)\sigma(s)^\top)^{-1}c(s) \int_0^s \gamma(s,u;s)c(u)^\top(\sigma(u)\sigma(u)^\top)^{-1}d\tilde{Y}_uds\\
    =&\int_0^t a(s)\tilde{\mu}_{s,s}ds+\int_0^t\gamma(u) c(u)^\top(\sigma(u)\sigma(u)^\top)^{-1}\{d\tilde{Y}_u-c(s)\tilde{\mu}_{s,s}ds\},
  \end{align*}
  which gives the equation
  \begin{align*}
    d\tilde{\mu}_{t,t}=a(t)\tilde{\mu}_{t,t}dt+\gamma(t) c(t)^\top(\sigma(t)\sigma(t)^\top)^{-1}\{d\tilde{Y}_t-c(t)\tilde{\mu}_{t,t}dt\}.
  \end{align*}
  By (\ref{eq3-12}) and $d E[X_t]=a(t)E[X_t]dt$, this is equivalent to the desired result.
\end{proof}

\begin{proposition}{\bf (Rauch-Tung-Striebel smoother)}\\
  If $\gamma(r)$ is strictly positive for every $r \in [0,t]$, it holds for every $0\leq s\leq t$
  \begin{align}\label{eq-RTS-mu}
    \mu_{s;t}=&\mu_{t;t}-\int_s^t \left\{ a(r)\mu_{r;t}+b(r)b(r)^\top \gamma(r)^{-1}(\mu_{r;t}-\mu_{r;r}) \right\}dr,
  \end{align}
  and
  \begin{align}\label{eq-RTS-gamma}
    \begin{split}
      \frac{d}{ds}\gamma(s,s;t)=&-b(s)b(s)^\top+\gamma(s,s;t)\left\{ a(s)^\top + \gamma(s)^{-1}b(s)b(s)^\top \right\}\\
    &+\left\{ a(s) + b(s)b(s)^\top\gamma(s)^{-1} \right\}\gamma(s,s;t).
      \end{split}    
  \end{align}
\end{proposition}
\begin{proof}
  Set 
  \begin{align*}
    \tilde{\mu}_{t;t}=\mu_{t;t}-E[X_t],~~\hat{Y}_t=\int_0^t c(u)^\top (\sigma(u) \sigma(u)^\top)^{-1} (dY_u - c(u) E[X_u] \, du).
  \end{align*}
  According to Proposition \ref{prop-expression-cov-xi}, for $s\leq u$ we have
  \begin{align*}
    d_s\gamma(s,u;t)
    =\left\{ a(s)+b(s)b(s)^\top \gamma(s)^{-1} \right\}\gamma(s,u;t)ds
  \end{align*}
  On the other hand, it follows from (\ref{eq-cov-xi-2}) that for $s\geq u$
  \begin{align}
    \label{eq-ds-cov-xi}d_s\gamma(s,u;t)=\left\{ a(s)+b(s)b(s)^\top \phi(s;t) \right\}\gamma(s,u;t)ds.
  \end{align}
  Therefore, it holds
  \begin{align*}
    \tilde{\mu}_{s;t}=&\int_0^t\left\{\gamma(t,u;t)-\int_s^t \frac{\partial}{\partial r}\gamma(r,u;t)dr \right\} d\hat{Y}_u\\
    =&\tilde{\mu}_{t;t}-\int_s^t \int_0^t \frac{\partial}{\partial r}\gamma(r,u;t)d\hat{Y}_u dr\\
    =&\tilde{\mu}_{t;t}-\int_s^t \int_0^r \left\{ a(r)+b(r)b(r)^\top \phi(r;t) \right\}\gamma(r,u;t)d\hat{Y}_u dr\\
    &-\int_s^t \int_r^t \left\{ a(r)+b(r)b(r)^\top \gamma(r)^{-1} \right\}\gamma(r,u;t)d\hat{Y}_u dr\\
    =&\tilde{\mu}_{t;t}-\int_s^t \left\{ a(r)+b(r)b(r)^\top \gamma(r)^{-1} \right\}\int_0^t \gamma(r,u;t)d\hat{Y}_u dr\\
    &+\int_s^t \int_0^r b(r)b(r)^\top \{\gamma(r)^{-1}-\phi(r;t)\} \gamma(r,u;t)d\hat{Y}_u dr\\
    =&\tilde{\mu}_{t;t}-\int_s^t \left\{ a(r)+b(r)b(r)^\top \gamma(r)^{-1} \right\}\tilde{\mu}_{r;t} dr\\
    &+\int_s^t b(r)b(r)^\top \int_0^r  \gamma(r,r;t)^{-1} \gamma(r,u;t)d\hat{Y}_u dr.
  \end{align*}
  In the last equation, we utilized the relation (\ref{eq-w-inverse-formula}).

  Furthermore, according to Proposition \ref{prop-expression-cov-xi}, $\gamma(r,r;t)^{-1} \gamma(r,u;t)$ does not depend on $t$. Hence, we have
  \begin{align*}
    \int_0^r  V[\xi_{r;t}]^{-1}\gamma(r,u;t)d\hat{Y}_u &=\int_0^r \gamma(r,r;r)^{-1} \gamma(r,u;r)d\hat{Y}_u=\gamma(r)^{-1}\tilde{\mu}_{r,r}.
  \end{align*}

  Therefore, we obtain
  \begin{align*}
    \tilde{\mu}_{s;t}=&\tilde{\mu}_{t;t}-\int_s^t \left\{ a(r)+b(r)b(r)^\top \gamma(r)^{-1} \right\}\tilde{\mu}_{r;t} dr+\int_s^t b(r)b(r)^\top \gamma(r)^{-1}\tilde{\mu}_{r,r} dr.
  \end{align*}
  Finally, we obtain (\ref{eq-RTS-mu}) by using $\tilde{\mu}_{s;t}=\mu_{s;t}-E[X_s]$ and $dE[X_s]=a(s)E[X_s]dt$. To prove (\ref{eq-RTS-gamma}), exchange $\gamma(s)^{-1}$ and $\phi(s;t)$ in (\ref{eq-derivative-w}), and recall $w(s;t)=V[\xi_{s;t}]=\gamma(s,s;t)$.
\end{proof}

\begin{proposition}\label{prop-derivative-cov-xi-s-u-t}
  It holds for $s,u\geq0$ and $t\geq s\vee u$ that
  \begin{align*}
    d_t\gamma(s,u;t)=-\gamma(t,s;t)^\top c(t)^\top(\sigma(t)\sigma(t)^\top)^{-1}c(t)
    \gamma(t,u;t)dt.
  \end{align*}
\end{proposition}
\begin{proof}
  First assume that $\gamma(t)$ is positive for every $t\geq 0$. Then, due to (\ref{eq-RTS-gamma}), it holds  
  \begin{align*}
    \frac{d}{ds}\gamma(s,s;t)=&-b(s)b(s)^\top+\gamma(s,s;t)\left\{ a(s)^\top + \gamma(s)^{-1}b(s)b(s)^\top \right\}\\
    &+\left\{ a(s) + b(s)b(s)^\top \gamma(s)^{-1} \right\}\gamma(s,s;t).
  \end{align*}
  Thus, we obtain for $s\leq t$
  \begin{align*}
    \gamma(s,s;t)=&\beta(t,s)\gamma(t) \beta(t,s)^{\top}+\int_s^t \beta(u,s)b(u) b(u)^\top\beta(u,s)^\top du,
  \end{align*}
  where $\beta$ is defined in (\ref{eq-def-beta}) and we used $V[\xi_{t;t}]=\gamma(t)$. Differentiating this yields
  \begin{align*}
    &d_t\gamma(s,s;t)
    =\beta(t,s)\\
    &\times \biggl[ -\left\{a(t)+b(t)b(t)^\top\gamma(t)^{-1} \right\}\gamma(t) -\gamma(t)\left\{ a(t)^\top+\gamma(t)^{-1}b(t)b(t)^\top\right\}\\
    &+\frac{d}{dt}\gamma(t)+b(t)b(t)^\top \biggr]\beta(t,s)\\
    =&-\beta(t,s)\gamma(t)c(t)^\top(\sigma(t)\sigma(t)^\top)^{-1}c(t)\gamma(t)\beta(t,s)t
  \end{align*}
  for $t\geq s$. Thus, noting that it follow from Proposition \ref{prop-expression-cov-xi} that $\gamma(t)\beta(t,s)^\top =\gamma(t,s;t)$, we obtain
  \begin{align}
    \label{eq3-13}\begin{split}      
    &d_t\gamma(s,u;t)
    =d_t\gamma(s,s;t)\beta(s,u)\\
    =&-\beta(t,s)\gamma(t)c(t)^\top(\sigma(t)\sigma(t)^\top)^{-1}c(t)\gamma(t)\beta(t,u)dt\\
    =&-\gamma(t,s;t)^\top c(t)^\top(\sigma(t)\sigma(t)^\top)^{-1}c(t)
    \gamma(t,u;t)dt
    \end{split}
  \end{align}
  for $0\leq u\leq s \leq t$. In the case where \( s \leq u \), we can derive the same formula by noting that $\frac{\partial}{\partial t} \mathrm{Cov}(\xi_{s; t}, \xi_{u; t})
  = \frac{\partial}{\partial t} \mathrm{Cov}(\xi_{u; t}, \xi_{s; t})^\top.$ Finally, assume that \( \gamma(t) \) may be singular. Then, we have
\begin{align*}
  d_t \mathrm{Cov}(\xi_{s; t}^\epsilon, \xi_{u; t}^\epsilon) 
  &= -\mathrm{Cov}(\xi_{t; t}^\epsilon, \xi_{s; t}^\epsilon)^\top c(t)^\top (\sigma(t) \sigma(t)^\top)^{-1} c(t) 
  \mathrm{Cov}(\xi_{t; t}^\epsilon, \xi_{u; t}^\epsilon) \, dt,
\end{align*}
where \( \gamma(s)^\epsilon \) and \( \xi_{s; t}^\epsilon \) are defined in (\ref{eq-def-gamma-epsilon-1}) to (\ref{eq-def-xi-epsilon-2}). Now, noting the fact that \( \mathrm{Cov}(\xi_{t; t}^\epsilon, \xi_{s; t}^\epsilon) \) converges uniformly to \( \mathrm{Cov}(\xi_{t; t}, \xi_{s; t})=\gamma(t,s;t) \) on any compact set as \( \epsilon \to 0 \), we have the desired result.

\end{proof}

\begin{proposition}{\bf (Fixed-point smoother)}\\
  For any $s \geq 0$, the process $\{\mu_{s;t}\}_{t\geq s}$ satisfies the stochastic differential equation
  \begin{align*}
    d_t\mu_{s;t}=\gamma(s,t;t)c(t)^\top(\sigma(t)\sigma(t)^\top)^{-1}(dY_t-c(t)\mu_{t;t}dt).
  \end{align*}
\end{proposition}
\begin{proof}
  By the definition of $\mu_{s;t}$ and Proposition \ref{prop-derivative-cov-xi-s-u-t}, we have
  \begin{align*}
    &d_t\mu_{s;t}\\
    =&\gamma(s,t;t)c(t)^\top(\sigma(t)\sigma(t)^\top)^{-1}(dY_t-c(t)E[X_t]dt)\\
    &-\int_0^t \gamma(t,s;t)^\top c(t)^\top(\sigma(t)\sigma(t)^\top)^{-1}c(t)
    \gamma(t,u;t)c(u)^\top(\sigma(u)\sigma(u)^\top)^{-1}\\
    &\times (dY_u-c(u)E[X_u]du)dt\\
  =&\gamma(s,t;t)c(t)^\top(\sigma(t)\sigma(t)^\top)^{-1}(dY_t-c(t)E[X_t]dt)\\
  &- \gamma(t,s;t)^\top c(t)^\top(\sigma(t)\sigma(t)^\top)^{-1}c(t)(\mu_{t;t}-E[X_t])dt\\
  =&\gamma(s,t;t)c(t)^\top(\sigma(t)\sigma(t)^\top)^{-1}(dY_t-c(t)\mu_{t;t}dt).
  \end{align*}
\end{proof}

\section{Further Applications}\label{section-application}
As mentioned above, Theorem~\ref{main-theorem} enables pathwise sampling from the conditional distribution of the hidden state process. This opens up a range of potential applications, including:
\begin{itemize}
  \item \textbf{Expectation of nonlinear functionals:} 
  The ability to sample from the full smoothing distribution enables the approximation of expectations involving nonlinear and path-dependent functionals, such as
  \[
    E\left[ \max_{0 \leq s \leq t} X_s \,\middle|\, \mathcal{Y}_t \right],
  \]
  which are intractable under standard smoothing equations. Such expectations are relevant in financial option pricing and engineering threshold detection.
  
  \item \textbf{Parameter estimation via the EM algorithm.}  
  The EM algorithm is a widely used method for maximum likelihood estimation when part of the data is latent. In the E-step, the expectation of the complete-data log-likelihood (Q-function) must be taken with respect to the smoothing distribution of the hidden states. Our main theorem allows for this expectation to be approximated via Monte Carlo simulation, enabling efficient estimation even in complex settings. 

  It can also be applied to nonlinear models in combination with linearization techniques, such as the Extended Kalman Filter (EKF). For instance, after linearizing the system dynamics and generating sample paths of the hidden state, the Q-function can be approximated without further linearizing the observation (transfer) function in the likelihood.
  
  \item \textbf{Construction of simultaneous confidence bands.} Unlike pointwise confidence intervals derived from smoothed covariance matrices, our pathwise sampling approach enables the construction of simultaneous confidence bands for the hidden process, ensuring global coverage over a time interval. This is particularly useful for detecting structural deviations or threshold exceedances across time.
\end{itemize}

\begin{appendix}
\section{Lemma for  (\ref{eq-conditional-of-Z})}
Let \( C_t \) be the space of continuous functions \( g: [0, t] \to \mathbb{R}^{d_2} \), and let \( \mathcal{C}_t \) be the cylindrical \(\sigma\)-algebra.
\begin{lemma}\label{lemma-stochastic-integral-measurable-function}
  Let $\{H_t\}_{t\geq 0}$ be a $d_1$-dimensional continuous measurable process, and write $\tilde{V}_s=\int_0^s b(s)dV_s.$
  Then, there exists a measurable function $\Phi_t:(C_t\times C_t,\mathcal{C}_t\otimes \mathcal{C}_t)\to \mathbb{R}$ such that
  \begin{align}
    \label{eq2-1}\Phi_t(H.,\tilde{V}.)=\int_0^t H_s d\tilde{V}_s~~\textrm{a.s.},
  \end{align}
  and for $\mu_H$-a.e. $h\in C_t$
  \begin{align}
    \label{eq2-2}\Phi_t(h,\tilde{V}.)=\int_0^t h(s) d\tilde{V}_s~~\textrm{a.s.}
  \end{align} 
  Here, $H.$ and $\tilde{V}.$ denote the paths of $\{H_s\}_{0\leq s\leq t}$ and $\{\tilde{V}_s\}_{0\leq s\leq t}$, and $\mu_H$ is the distribution of $H.$ on $C_t$.
\end{lemma}
\begin{proof}
  First we can take a sequence $\{n(k)\}_{k \in \mathbb{N}}$ such that
  \begin{align}
    \label{eq2-3}\sum_{i=1}^{n(k)} H_{t/{2^{n(k)-1}}}^\top(\tilde{V}_{t/{2^{n(k)}}}-\tilde{V}_{t/{2^{{n(k)}-1}}})\xrightarrow{\textrm{a.s.}} \int_0^t H_s d\tilde{V}_s.
  \end{align}
  Now we introduce a sequence of measurable functions \( \Phi_t^k : C_t \times C_t \to \mathbb{R} \) for \( k \in \mathbb{N} \), defined by
\begin{align*}
    \Phi_t^k(h, w) = \sum_{i=1}^{n(k)} h\left(\frac{t}{2^{n(k)-1}}\right)^\top \left(w\left(\frac{t}{2^{n(k)}}\right) - w\left(\frac{t}{2^{n(k)-1}}\right)\right).
\end{align*}
We then define \( \Phi_t : C_t \times C_t \to \mathbb{R} \) by
\begin{align*}
    \Phi_t(h, w) = \begin{cases}
      \displaystyle \lim_{k \to \infty} \Phi_t^k(h, w) & \text{if } \Phi_t^k(h, w) \text{ converges,} \\
      0 & \text{otherwise}.
    \end{cases}
\end{align*}
Since each \( \Phi_t^k \) is measurable, \( \Phi_t \) is also measurable, and (\ref{eq2-1}) follows directly from the construction of \( \Phi_t \).

  Let us verify (\ref{eq2-2}). If we set $A=\left\{ (h,w)\in C_t\times C_t; \Phi_t^k(h,w)~\textrm{converges~to~}\Phi_t(h,w)  \right\}$, then (\ref{eq2-3}) means
  \begin{align*}
    \int_{C_t\times C_t}1_{A^c}(h,w)d\mu_{(H,\tilde{V})}(h,w)=0,
  \end{align*}
  where $\mu_{(H,\tilde{V})}$ is the distribution of $(H.,\tilde{V}.)$. Due to the independence of $H.$ and $\tilde{V}.$, this can be rewritten as
  \begin{align*}
    \int_{C_t}\int_{C_t}1_{A^c}(h,w)d\mu_{\tilde{V}}(w)d\mu_{H}(h)=0,
  \end{align*}
  and thus we have for $\mu_H$-a.e. $h\in C_t$,
  \begin{align*}
    \int_{C_t}1_{A^c}(h,w)d\mu_{\tilde{V}}(w)=0.
  \end{align*}
  From this, it follows that for $\mu_H$-a.s. $h\in C_t$, the sequence of random variables
  \begin{align}
    \label{eq2-4}\Phi_t^k(h,W.)=\sum_{i=1}^{n(k)} h({t/{2^{n(k)-1}}})^\top(\tilde{V}_{t/{2^{n(k)}}}-\tilde{V}_{t/{2^{{n(k)}-1}}})
  \end{align}
  almost surely converges to $\Phi_t(h,\tilde{V})$. On the other hand, since \( \{H_s\}_{0 \leq s \leq t} \) is continuous, we may assume that \( h \) is continuous. For such \( h \), the right-hand side of (\ref{eq2-4}) converges in probability to \( \int_0^t h(s) \, d\tilde{V}_s \), thereby yielding (\ref{eq2-2}).
\end{proof}

\section{Proofs of Lemma \ref{lemma-quadratic-form} and (\ref{eq-tensor-expectation})}\label{section-expectation-calculation}
  Using (\ref{eq-rewrite-exp}), we obtain
  \begin{align}
    \label{eq-z-y-1}\begin{split}
      E_Q[z_t(y)]=E\biggl[ \exp\biggl(&({\xi^0_{t;t}})^\top y(t)- \int_0^ty(s)^\top d{\xi^0_{s;t}}\\
      &+\frac{1}{2}\left. \left.\int_0^t b(s)^\top \phi(s;t)b(s)ds+\frac{1}{2}X_0^\top \phi(0;t) X_0\right) \right].
    \end{split}    
  \end{align}
  Furthermore, substituting the decomposition (\ref{eq-decompisition-xi0}) into (\ref{eq-z-y-1}), we can write
  \begin{align}
     \label{eq3-7}\begin{split}
      &E_Q[z_t(y)] =\exp\left( \frac{1}{2}\int_0^t b(s)^\top \phi(s;t)b(s)ds \right)\\
     &\times E\biggl[ \exp\biggl(X_0^\top \alpha(t;t)^\top y(t)\\
     &\left. \left.- \int_0^ty(s)^\top \left\{ a(s)+b(s)b(s)^\top \phi(s;t) \right\}\alpha(s;t)ds X_0+\frac{1}{2}X_0^\top \phi(0;t) X_0\right) \right]\\
     &\times E\left[ \exp\left(\overline{\xi}_{t;t}^\top y(t)- \int_0^ty(s)^\top d\overline{\xi}_{s;t}\right) \right].
     \end{split}     
  \end{align}

  For the first expectation on the right-hand side, let \( X_0 = \mu + \Sigma^{\frac{1}{2}} U \), where \( U \) is a standard normal random variable, \( \mu \in \mathbb{R}^{d_1} \), and \( \Sigma \in M_{d_1}(\mathbb{R}) \). Then, by straightforward calculation, we obtain
  \begin{align}
    &E\biggl[ \exp\biggl(X_0^\top \alpha(t;t)^\top y(t)\\
    &\left. \left.- \int_0^ty(s)^\top \left\{ a(s)+b(s)b(s)^\top \phi(s;t) \right\}\alpha(s;t)ds X_0+\frac{1}{2}X_0^\top \phi(0;t) X_0\right) \right]\nonumber\\
     \label{eq3-8}\begin{split}
      =&\sqrt{\det|I_{d_1}-\Sigma^\frac{1}{2}\phi(0;t)\Sigma^\frac{1}{2}|}\exp\biggl(\frac{1}{2}\beta(t;y)^\top \Sigma^\frac{1}{2}(I_{d_1}-\Sigma^\frac{1}{2}\phi(0;t)\Sigma^\frac{1}{2})^{-1}\Sigma^\frac{1}{2}\beta(t;y) \biggr)\\
     &\times \exp\biggl(\mu^\top \alpha(t;t)^\top y(t)\\
     &- \int_0^ty(s)^\top \left\{ a(s)+b(s)b(s)^\top \phi(s;t) \right\}\alpha(s;t)ds \mu+\frac{1}{2}\mu^\top \phi(0;t) \mu\biggr).
     \end{split}     
  \end{align}  
  Here, we used the negative semi-definiteness of $\phi(0;t)$.

  For the second expectation, since $ \overline{\xi}_{t;t}^\top y(t)- \int_0^ty(s)^\top d\overline{\xi}_{s;t}$ is a Gaussian random variable with mean 0, it follows that 
  \begin{align}
    &E\left[ \exp\left(\overline{\xi}_{t;t}^\top y(t)- \int_0^ty(s)^\top d\overline{\xi}_{s;t}\right) \right]\nonumber
    =\exp\left( \frac{1}{2}E\left[ \left( \overline{\xi}_{t;t}^\top y(t)- \int_0^ty(s)^\top d\overline{\xi}_{s;t} \right)^2 \right] \right)\nonumber\\
    \label{eq3-9}\begin{split}      
    =&\exp\biggl( \frac{1}{2}y(t)^\top E[\overline{\xi}_{t;t}\overline{\xi}_{t;t}^\top]y(t)\\
    &+\frac{1}{2}\int_0^t y(s)^\top b(s)b(s)^\top y(s)ds-\int_0^t y(s)^\top b(s)b(s)^\top y(t) ds  \biggr).
    \end{split}
  \end{align}
  Putting (\ref{eq3-7}), (\ref{eq3-8}) and (\ref{eq3-9}) together, we obtain Lemma \ref{lemma-quadratic-form}. In the same way, (\ref{eq-tensor-expectation}) is shown by
  \begin{align*}
    &E_Q[ X_{u_1}\otimes \cdots \otimes X_{u_{n}}z_t(y)]\nonumber\\
    =&E\biggl[\xi^0_{u_1;t}\otimes\cdots \otimes \xi^0_{u_{n};t} \exp\biggl({\xi^0_{t;t}}^\top y(t)- \int_0^ty(s)^\top d{\xi^0_{s;t}}\\
    &+\frac{1}{2}\left. \left.\int_0^t b(s)^\top \phi(s;t)b(s)ds+\frac{1}{2}X_0^\top \phi(0;t) X_0\right) \right]\\
     \begin{split}
      =&\exp\left( \frac{1}{2}\int_0^t b(s)^\top \phi(s;t)b(s)ds \right) E\biggl[\{\alpha(u_1;t)X_0+\overline{\xi}_{u_1;t}\}\otimes \cdots \otimes \{\alpha(u_1;t)X_0+\overline{\xi}_{u_1;t}\}\\
     &\times  \exp\biggl(X_0^\top \alpha(t;t)^\top y(t)\\
     &- \int_0^ty(s)^\top \left\{ a(s)+b(s)b(s)^\top \phi(s;t) \right\}\alpha(s;t)ds X_0+\frac{1}{2}X_0^\top \phi(0;t) X_0\biggr)\\
     &\left.\times  \exp\left(\overline{\xi}_{t;t}^\top y(t)- \int_0^ty(s)^\top d\overline{\xi}_{s;t}\right) \right].
     \end{split}     
  \end{align*}  

\end{appendix}

\section*{Acknowledgement}
The author is deeply grateful to N. Yoshida for his valuable advice and insightful discussions.

\section*{Funding}
This work was supported by Japan Science and Technology Agency CREST JP-MJCR2115, JSPS KAKENHI Grant Number JP24KJ0667, and RIKEN Special Postdoctoral Researcher Program.

\bibliographystyle{abbrvnat}

\bibliography{arxiv2}

\end{document}